\documentclass[12pt]{amsart}
\usepackage{amsfonts, amsthm, amsmath, amssymb}
\usepackage[pagebackref=true]{hyperref}
\usepackage{mathtools}
\usepackage[square]{natbib}
\hypersetup{colorlinks=false}

\usepackage{color}

\newtheorem{theo}{Theorem}[section]
\newtheorem{cor}[theo]{Corollary}
\newtheorem{prop}[theo]{Proposition}
\newtheorem{lemma}[theo]{Lemma}

\theoremstyle{remark}
\newtheorem{remark}[theo]{Remark}

\theoremstyle{definition}
\newtheorem{defi}[theo]{Definition}

\numberwithin{equation}{section}

\def \tr {\operatorname{tr}}
\def \Frob {\operatorname{Frob}}
\def \Res {\operatorname{Res}}
\def \res {\operatorname{res}}

\def \Gal {\operatorname{Gal}}

\def \Ql {\overline{\mathbb Q}_\ell}

\def \FFrob {\operatorname{Fr}}
\newcommand \nchi {n}
\begin{document}

\title{General Multiple Dirichlet Series from Perverse Sheaves}
\author{Will Sawin}
\address{Department of Mathematics \\ Columbia University \\ New York, NY}
\email{sawin@math.columbia.edu}

\maketitle

\begin{abstract} We give an axiomatic characterization of multiple Dirichlet series over the function field $\mathbb F_q(T)$, generalizing a set of axioms given by Diaconu and Pasol. The key axiom, relating the coefficients at prime powers to sums of the coefficients, formalizes an observation of Chinta. The existence of multiple Dirichlet series satisfying these axioms is proved by exhibiting the coefficients as trace functions of explicit perverse sheaves and using properties of perverse sheaves. The multiple Dirichlet series defined this way include, as special cases, many that have appeared previously in the literature. \end{abstract}

\tableofcontents

\section{Introduction}

\subsection{Background}

Multiple Dirichlet series were originally defined as Dirichlet series in multiple variables satisfying twisted muliplicativity properties and certain groups of functional equations. These were first motivated by moments of $L$-functions \citep{Siegel1955,HG85}, and have since been successfully used to calculate several moments, with recent examples including \citep*{Diaconu2019, DW,GaoZhao, GaoZhao2}. If one defines a Dirichlet $L$-function where the Dirichlet character is expressed as a Legendre symbol, as in
\[ L \left(s, \left( \frac{\cdot}{m} \right)\right) = \sum_{n=1}^{\infty} \left(\frac{n}{m} \right) n^{-s} \]
then it is natural to consider moments like \[ \sum_{m < X} \prod_{i=1}^kL \left(s_i, \left( \frac{\cdot}{m} \right)\right), \] which can be analyzed using the series \[ \sum_{m=1}^\infty \prod_{i=1}^kL \left(s_i, \left( \frac{\cdot}{m} \right)\right) m^{-s} = \sum_{n_1,\dots, n_k, m=1}^{\infty}  \Bigl(\prod_{i=1}^k \left(\frac{n_i }{m} \right)    \Bigr)m^{-s} \prod_{i=1}^k n_i^{-s_i}.\] A plausible strategy to analyze these moments is as follows. First, replace the coefficients $\prod_{i=1}^k \left(\frac{n_i }{m} \right)$ by another set of coefficients $a_{n_1,\dots,n_k, m }$ which agree with them for $n_1,\dots, n_k, m$ squarefree and relatively prime, but may differ for other values. Also choose the coefficients $a_{n_1,\dots,n_k, m }$ to ensure the series has better analytic properties. Next, use these analytic properties to estimate suitable integrals of the series. Finally, use a sieve to extract information about the corresponding integral with the original set of coefficients. Since the coefficients $\prod_{i=1}^k \left(\frac{n_i }{m} \right)$ satisfy a twisted multiplicativity analogous to the multiplicativity of the coefficients of classical Dirichlet series, one assumes the modified coefficients keep this twisted multiplicativity, i.e. \[ a_{n_1n_1', \dots, n_k n_k', mm'} =a_{n_1,\dots,n_k, m } a_{n_1',\dots,n_k', m' } \prod_{i=1}^k \left( \frac{n_i}{m'} \right) \left( \frac{n_i'}{m} \right) \] as long as $n_1,\dots,n_k,m$ are relatively prime to $n_1',\dots,n_k',m'$. Generally, the better analytic properties one seeks to obtain are functional equations, and analytic continuation enabled by those functional equations.

The most desirable would be a meromorphic continuation to $\mathbb C^k$ with an explicit description of the poles. This can be obtained when one has a functional equation in each variable generating a finite group of functional equations (typically a Weyl group). However, some recent work has studied multiple Dirichlet series with an infinite group of functional equations \citep{Whitehead}, where one expects only meromorphic continuation to a certain region in $\mathbb C^r$, and can only prove meromorphic continuation to a smaller region directly from the functional equations. Still, obtaining continuation to the larger region is sometimes possible \citep{Whitehead2}, and could hold the key to estimating higher moments of $L$-functions \citep*{DGH,DiaconuTwiss}.

Since the multiplicativity is twisted, one does not have an expression of the multiple Dirichlet series as an Euler product of local factors. However, twisted multiplicativity does still reduce the choice of coefficients for each tuple of numbers to the local choice of coefficients for each tuple of powers of a fixed prime. To obtain the desired functional equations, one needs that the generating series of these prime power coefficients satisfy certain analogous functional equations. Because these local functional equations were used to define the coefficients, the multiple Dirichlet series could only be uniquely defined when these functional equations were sufficient to uniquely characterize the generating functions. \citet{Chinta} first observed that, when working over the function field $\mathbb F_q(t)$, there was a local-to-global symmetry relating these generating functions to the multiple Dirichlet series. This could be proven by observing that they were both determined by their functional equations and then comparing their functional equations.

\subsection{Summary of results}

The goal of this paper is to provide a uniform construction of multiple Dirichlet series over the function field $\mathbb F_q(t)$, parameterized simply by the finite field $\mathbb F_q$, a character $\chi$ of $\mathbb F_q$, and a symmetric integer matrix $M$, that includes many multiple Dirichlet series previously constructed separately as well as new examples. Future work could investigate these new examples, finding functional equations they satisfy, regions to which they can be analytically continued, and applications to moments of $L$-functions. Furthermore, it may be possible to define new multiple Dirichlet series in the number field context by choosing the coefficients at tuples of powers of a prime $p$ to match the coefficients of the series defined here at powers of a polynomial over $\mathbb F_p$, and then to investigate their analytic properties also. 

Our approach is inspired by \citet{DiaconuPasol}, who showed that the local-to-global properties observed by \citet{Chinta}, combined with the twisted multiplicativity, uniquely characterize the multiple Dirichlet series by an inductive argument, and thus could be used as a definition of multiple Dirichlet series. However, they were only able to show the existence of multiple Dirichlet series satisfying these local-to-global properties in one particular family of cases, the one relating to moments of quadratic Dirichlet $L$-functions, by a lengthy \'{e}tale cohomology argument. In these cases, \citet{Whitehead} was able to show that the functional equations follow from the local-to-global properties.

We propose a new approach. We define multiple Dirichlet series that satisfy quite general twisted multiplicativity relations involving arbitrary characters, which are uniquely characterized by local-to-global properties. Here the matrix $M$ and character $\chi$ determine the exact function we twist the multiplicativity relation by. However, we define and construct the multiple Dirichlet series coefficients as trace functions of certain perverse sheaves.

Using this local-to-global property, it is possible to show that our multiple Dirichlet series include as a special case some multiple Dirichlet series that appear before in the literature. We prove this for two series defined by \citet{ChintaMohler} (Corollary \ref{easiest-example-cm-2} and \eqref{compare-to-cm}) and one defined by \citet{Chinta} (Proposition \ref{double-gauss-example}). For those defined by \cite{DiaconuPasol} the proof is automatic since their axioms are a special case of ours. It seems reasonable to expect, based on these examples, that every multiple Dirichlet series defined in the literature whose values at relatively prime tuples of squarefree numbers can be expressed in terms of Dirichlet characters, Jacobi symbols, and Gauss sums, are also special cases of our construction, while those expressed using Fourier coefficients of higher rank automorphic forms, as summarized in \citep*{CFH}, are not. However, it is very plausible that multiple Dirichlet series related to higher rank automorphic forms could arise from perverse sheaves constructed in a similar way using the Langlands parameter of the automorphic form. In addition to the examples, these expectations are motivated by the idea that the trace function of a perverse sheaf gives the best way to extend a function from ``generic" values like tuples of relatively prime squarefree numbers to all values, and therefore that every extension that satisfies nice analytic properties likely comes from a suitable perverse sheaf.

The idea that the trace function of a perverse sheaf gives a well-behaved function in analytic number theory over function fields is most prominent in the geometric Langlands program, where automorphic forms are expected, and in many cases known, to arise in this way, but it can also be seen in more elementary situations. For example, the divisor function arises from a perverse sheaf. More generally, so do the coefficients of the $L$-function of a Galois representation.

The author also expects that these multiple Dirichlet series will satisfy functional equations analogous to those satisfied by existing series like the Weyl group multiple Dirichlet series \citep*{BBCFH}, and possibly more general ones, with the exact nature of the functional equations depending on the parameters $M,\chi$. The examples included in this paper give initial evidence for this: Proposition \ref{double-gauss-example} covers a Weyl group multiple Dirichlet series that satisfies an interesting group of functional equations matching the Weyl group $S_3$, suggesting that further special cases of our construction may also satisfy similar functional equations. Furthermore \eqref{fourier-for-fe} gives a relation between the coefficients of two multiple Dirichlet series that can be used to prove a functional equation relating the series themselves, with the Fourier transform in that equation playing the same crucial role it does in the classical functional equations of the zeta function and Dirichlet $L$-functions, again suggesting that more general functional equations of this type should exist. Work in progress by the author and Ian Whitehead, as well as by Matthew Hase-Liu, aims to prove these functional equations in greater generality. This work should also enable us to realize further previously-defined multiple Dirichlet series as special cases of the construction of this paper, as these series are uniquely determined by their functional equations so it suffices to check the newly-defined series satisfy the same functional equations.


\subsection{Notation}

Let $\mathbb F_q[t]$ be the ring of polynomials in one variable over a finite field $\mathbb F_q$. Let $\mathbb F_q[t]^+$ be the subset of monic polynomials. Let $f'$ be the derivative of $f$ with respect to $t$.

Fix a natural number $n$. We always let $\chi\colon \mathbb F_q^\times \to \mathbb C^\times$ be a character of order $n$.  Let $\chi_m\colon \mathbb F_{q^m}^\times \to \mathbb C^\times$ be the composition of $\chi$ with the norm map $\mathbb F_{q^m} \to \mathbb F_q$.

Define a residue symbol \[ \left( \frac{f}{g} \right)_\chi \] for $(f,g) \in \mathbb F_q[t]$ coprime as the unique function that is separately multiplicative in $f$ and $g$ such that if $g$ is irreducible of degree $d$, \[ \left( \frac{f}{g} \right)_\chi  = \chi \bigl( f^{ \frac{q^d-1}{q-1}} \bigr) ,\] where we use the fact that $f^{ \frac{q^d-1}{q-1}}$ in $\mathbb F_q[T]/g = \mathbb F_{q^d}$ in fact lies in $\mathbb F_q$.

Let $\Res(f,g)$ be the resultant of $f$ and $g$. For $g$ monic, as it will almost always be, this is the product of the values of $f$ at the roots of $g$.

We define a ``set of ordered pairs of Weil numbers and integers" to be a set $J$ consisting of ordered pairs $j$ of a Weil number $\alpha_j$ and an integer $c_j$, such that no $\alpha_j$ appears twice in the set, and $c_j$ is never zero.

For $J_1, J_2$ two sets of ordered pairs, we define  $J_1 \cup J_2$ to be the union, except that if some Weil number $\alpha$ appears in both $J_1$ and $J_2$, we add the $c_j$s together, and if the sum is zero, we remove them. In other words, $J_1\cup J_2$ is the unique set of ordered pairs of Weil numbers and integers such that  \[\sum_{j \in J_1 \cup J_2} c_j \alpha_j^e = \sum_{j \in J_1} c_j \alpha_j^e + \sum_{j \in J_2} c_j \alpha_j^e\] for all integers $e$.

For a Weil number $\beta$, we take $\beta J$ to be the set of ordered pairs $ ( \beta \alpha_j, c_j)$, so that $\sum_{j \in \beta J} c_j \alpha_j^e = \beta^e \sum_{j \in J} c_j \alpha^e$ for all integers $e$.

We say a $\mathbb C$-valued function $\gamma(q,\chi)$ on pairs of a prime power $q$ and character $\chi$ of $\mathbb F_q^\times$ is a compatible system of Weil numbers if \[\gamma(q^e, \chi_e)= \gamma(q,\chi)^e\] for all $q,\chi,e$. For instance, the constant function $1$ is a compatible system of Weil numbers.

We say that a function $J(q,\chi)$ from pairs of a prime power $q$ and a character $\chi$ of $\mathbb F_q^\times$ to sets of ordered pairs of Weil numbers and integers is a compatible system of sets of ordered pairs if, whenever $J(q,\chi) =\{ (\alpha_j,c_j)\}$, we have $J (q^e, \chi_e) = \{ (\alpha_j^e, c_j)\}$, so that \[\sum_{j \in J (q^e, \chi_e) } c_j \alpha_j^r = \sum_{j \in J(q, \chi)} c_j \alpha_j^{re}.\]

 We now define the general construction of sheaves that will be key for our paper. Fix once and for all a prime $\ell$ invertible in $\mathbb F_q$ and an isomorphism between $\Ql $ and $\mathbb C$ (or just the fields of algebraic numbers within each), with which we will freely identify elements of $\Ql$ and $\mathbb C$. Let $X$ be an irreducible scheme of finite type over a field in which $\ell$ is invertible, generically smooth of dimension $d$, and $f$ a nonvanishing function on $X$. Let $U$ be the maximal smooth open set where $f$ is invertible and let $j\colon U \to X$ be the open immersion. We have a Kummer map $H^0(U, \mathbb G_m) \to H^1(U, \mu_{q-1} )$. The image of $f$ under this map defines a $\mu_{q-1} $-torsor. We can twist the constant sheaf $\Ql$ by the image of this torsor under $\chi\colon \mu_{q-1} = \mathbb F_q^\times \to \Ql^\times$, obtaining a lisse rank one sheaf $\mathcal L_\chi(f)$ on $U$. Because $U$ is smooth of dimension $d$, $\mathcal L_\chi (f)[d]$ is a perverse sheaf on $U$. Let $j_{*!} (\mathcal L_\chi(f) [d])$ be its middle extension from $U$ to $X$. Let \[ IC_{\mathcal L_\chi (f)} = j_{*!} (\mathcal L_\chi(f) [d])[-d]\] be this middle extension, shifted so it lies generically in degree zero.

\subsection{Construction and main theorem}

Let $r$ be a natural number and let $M$ be a symmetric $r \times r$ matrix with integer entries.

Let $d_1,\dots,d_{r}$ be natural numbers and $q$ a prime power so that $\mathbb F_q$ is a finite field. View $\mathbb A^{d_i}$ over $\mathbb F_q$ as the moduli space of monic polynomials of degree $d_i$, so that $\prod_{i=1}^{r} \mathbb A^{d_i} $ is a moduli space of tuples $(f_1,\dots,f_{r})$ of monic polynomials. On $\prod_{i=1}^{r} \mathbb A^{d_i}$, define the polynomial function \[F_{d_1,\dots,d_r}= \prod_{i=1}^r \Res (f_i', f_i) ^{M_{ii} }  \prod_{1\leq i< j \leq r}  \Res(f_i, f_j)^{ M_{ij}}.\]

Let \[K_{d_1,\dots,d_r}= IC_{\mathcal L_\chi(F_{d_1,\dots,d_{r}})}.\] Given a tuple of polynomials $(f_1,\dots,f_{r})$ of degrees $d_1,\dots, d_{r}$  over $\mathbb F_q$, let $a(f_1,\dots,f_{r}; q, \chi, M)$ be the trace of Frobenius acting on the stalk of $K_{d_1,\dots,d_{r}}$ at $(f_1,\dots,f_{r})$.

Define the multiple Dirichlet series \[Z(s_1,\dots,s_{r};q,\chi,M) = \sum_{f_1,\dots,f_{r} \in \mathbb F_q[t]^+} \frac{a(f_1,\dots,f_{r}; q, \chi,M) }{ \prod_{i=1}^{r} q^{- (\deg f_i) s_i}}.\]

The main theorem of this paper, giving an axiomatic characterization of the coefficients of the geometrically defined multiple Dirichlet series $Z(s_1,\dots,s_{r};q,\chi,M)$, is as follows.

\begin{theo}\label{axiomatics} For any fixed $M$, \[ a (f_1,\dots, f_r; q ,\chi,M) \] is the unique function, that, together with a function $J(d_1,\dots,d_r; q,\chi,M)$ from tuples of natural numbers $d_1,\dots, d_r$, to compatible systems of sets of ordered pairs of Weil numbers and integers, satisfies the axioms

\begin{enumerate}

\item If $f_1,\dots,f_{r}$ and $g_1,\dots, g_{r}$ satisfy $\gcd(f_i,g_j)=1$ for all $i$ and $j$, then we have \[ a ( f_1g_1,\dots, f_{r} g_{r}; q , \chi, M) \] \[= a(f_1,\dots,f_{r} ; q, \chi, M) a(g_1,\dots, g_{r} ; q, \chi, M)  \prod_{1 \leq i \leq r} \left( \frac{f_i}{g_i} \right)_{\chi}^{M_{ii}} \left( \frac{g_i}{f_i} \right)_{\chi}^{M_{ii}}   \prod_{1 \leq i < j \leq r}  \left( \frac{f_i}{g_j} \right)_{\chi}^{M_{ij}}\left( \frac{g_i}{f_j} \right)_{\chi}^{M_{ij}} .\] 

\item $a(1,\dots,1;q,\chi,M)=1$ and $a(1,\dots, 1, f, 1,\dots,1;q,\chi,M)=1$ for all linear polynomials $f$. 

\item 
   \[ a(\pi^{d_1},\dots, \pi^{d_r}; q, \chi, M) =\left(\frac{\pi'}{\pi}\right)_{\chi} ^{ \sum_{i=1}^r { d_i M_{ii}} }   \sum_{j \in J(d_1,\dots,d_r; q,\chi, M)}  c_j \alpha_j^{\deg \pi}.\]
   
  \item  
\[ \sum_{\substack{f_1,\dots,f_{r} \in \mathbb F_q[t]^+\\ \textrm{deg }f_i= d_i}}  a(f_1,\dots,f_{r}; q, \chi,M)= \sum_{j \in J(d_1,\dots,d_r; q,\chi, M)}  c_j  \frac{ q^{ \sum_{i=1}^{r} d_i}}{\overline{\alpha}_j }.\]
 
\item  $|\alpha_j| < q^{ \frac{  \sum_{i=1}^{r} d_i -1 }{2}}$ as long as $\sum_{i=1}^{r} d_i \geq 2$. 
 
 \end{enumerate}
  
 \end{theo} 

Here axioms (3) and (4) give the local-to-global principle, (1) is the twisted multiplicativity, and (2) and (5) are normalizations needed to ensure the axioms define a unique set of coefficients, with (5) also ensuring that individual coefficients are not so large that they dominate the series.

Note that the condition that $J$ be a compatible system relates different finite fields at a time, so it is not possible to check these axioms working only in a specific finite field $q$. Rather, one must calculate in all extensions of a fixed finite field $\mathbb F_{q_0}$.

In the case, $\chi$ is quadratic, when $M$ is the sum of a matrix with a row of ones and the rest of the entries zero and its transpose, the existence and uniqueness parts of Theorem \ref{axiomatics} were obtained in \citep{DiaconuPasol}.

\subsection{Perverse sheaves}

The key geometric idea of this paper is that the local-to-global property described by axioms (3) and (4) is a consequence of duality properties of perverse sheaves. The local-to-global property relates the sum of many coefficients of the multiple Dirichlet series to a single coefficient, via the set of Weil numbers $J$. Geometrically, we interpret this as a relation between the sum of the trace of Frobenius on the stalk of a perverse sheaf over all the $\mathbb F_q$-points of a variety and the value at a single point. The Lefschetz fixed point formula relates the sum of the trace of Frobenius over all $\mathbb F_q$-points to the compactly supported cohomology of the variety with coefficients in the perverse sheaf. Because there is an action of the multiplicative group on the variety that fixes only that point, giving it a conical structure, a generalization of the result that the cohomology of a cone matches the cohomology of the point relates the stalk of that point to the usual cohomology. Verdier duality for perverse sheaves then relates the usual and compactly-supported cohomology.

Furthermore, axiom (1) will follow from a twisted multiplicativity property of the polynomial functions $F_{d_1,\dots, d_r}$ used to construct the perverse sheaves $K_{d_1,\dots, d_r}$. We then transform this identity involving the polynomials $F_{d_1,\dots, d_r}$ to an isomorphism involving the perverse sheaves $K_{d_1,\dots, d_r}$, using fundamental properties of the intermediate extension construction, which then implies an identity involving the trace functions $a(f_1,\dots, f_r; q,\chi ,M)$ of the perverse sheaves $K_{d_1,\dots, d_r}$.

Axiom (5) follows from the theory of weights and purity for perverse sheaves, which gives bounds for the Frobenius eigenvalues in each degree.

Characteristic zero analogues of the perverse sheaves $ IC_{\mathcal L_\chi(F_{d_1,\dots,d_{r}})}$ used in our construction have been studied before from the perspective of quantum groups and Nichols algebras~\citep*{BFS,KS}. Some of our (brief) calculations with these sheaves in Section \ref{axioms} are characteristic $p$ analogues of results previously obtained in the characteristic zero setting in those works. This connection between multiple Dirichlet series and quantum groups seems different from the usual one, as the coefficients of the multiple Dirichlet series correspond to traces of Frobenius on stalks of the sheaves that can be computed from the cohomology of the positive part of the small quantum group, but neither the Frobenius action nor the cohomology of the positive part appear in the usual picture. I learned of these connections thanks to helpful conversations with Jordan Ellenberg, Michael Finkelberg, Mikhail Kapranov, Tudor P\u{a}durariu, and Vadim Schechtman.

While writing this paper, at various times the author served as a Clay Research Fellow, was supported by NSF grant DMS-2101491, and was a Sloan Research Fellow.
 I would like to thank Adrian Diaconu for helpful conversations, Matthew Hase-Liu and River Sawin for helping me find typos, and the anonymous referee for many helpful comments.

\section{Preliminaries}

\subsection{Further notations}

We use $\xi$ to refer to, when $q$ is odd, the unique character $\xi\colon \mathbb F_q^\times \to \mathbb C^\times$ of order $2$. If $n$ is even, we have $\xi=\chi^{n/2}$.

For a rational function $f$, let $\res(f)$ be its residue at $\infty$, normalized so that $\res(1/t)=1$, (i.e. the coefficient of $t^{-1}$ when $f$ is expressed as a formal Laurent series in $t^{-1}$). 

For $x \in \mathbb F_q$, let $\psi(x) = e^{ 2 \pi i \tr_{\mathbb F_q}^{\mathbb F_p} x /p}$. Let $G(\chi,\psi) = \sum_{x \in \mathbb F_q^\times} \chi(x) \psi(x)$. Let \[ g_\chi ( f_1,f_2) = \sum_{h \in \mathbb F_q[t]/ f_2} \left( \frac{h}{f_2} \right)_\chi \psi \left( \operatorname{res} \left(\frac{h f_1}{ f_2} \right) \right).\]

We say a function $\gamma(q,\chi)$ on pairs of a prime power $q$ and character $\chi$ of $\mathbb F_q^\times$ is a sign-compatible system of Weil numbers if \[-\gamma(q^e, \chi_e)= (-\gamma(q,\chi))^e\] for all $q,\chi,e$. For instance, the Hasse-Davenport identities imply that $G(\chi^r,\psi)$ is sign-compatible for any integer $r$.

We let \[\lambda(d_1,\dots,d_{r};q, \chi, M ) = \sum_{\substack{f_1,\dots,f_{r} \in \mathbb F_q[t]^+\\ \textrm{deg }f_i= d_i}}  a(f_1,\dots,f_{r}; q, \chi,M)\] so that \[ Z(s_1,\dots,s_{r} ;q,\chi, M)= \sum_{d_1,\dots,d_{r} \in \mathbb N} \frac{ \lambda(d_1,\dots,d_{r};q,\chi,M)}{ \prod_{i=1}^{r} q^{-d_i s_i}}.\]

For $\pi$ a prime polynomial, we let $v_\pi$ be the $\pi$-adic valuation of polynomials, i.e. $v_\pi(f)$ is the maximum power of $\pi$ dividing $f$.

\subsection{Function field evaluations}

Certain functions important in classical number theory, such as the M\"{o}bius function, power residue symbol, and Gauss sums, admit alternate formulas in the function field $\mathbb F_q(t)$, that make clear their relationship to the algebra of polynomials.

\begin{lemma}\label{resultant} We have \[ \left( \frac{f}{g} \right)_\chi  = \chi (\Res(f,g)) .\] \end{lemma}

\begin{proof} Because the right side, by definition, is multiplicative in $g$, it suffices to consider the case where $g$ is prime. Then for $\alpha$ a root of $g$, the other roots are $\alpha^q, \dots, \alpha^{q ^{d-1}} $. Hence the product of the values of $f$ at these roots is \[ \prod_{i=0}^{d-1} f( \alpha^{q^i}) = \prod_{i=0}^{d-1} f(\alpha)^{q^i} = f(\alpha)^{ \frac{ q^{d}-1}{q-1}} .\] Because $\alpha$ is a root of $g$, we can evaluate this by setting $\alpha=T$ and reducing mod $g(T)$, which matches the definition of $\left( \frac{f}{g} \right)_\chi $. \end{proof}

Under this interpretation, the reciprocity law for power residue symbols is given by the following fact:

\begin{lemma}\label{reciprocity} For monic $f,g$, \[\Res(f,g) = (-1)^{\deg f \deg g} \Res(g,f).\] \end{lemma}

\begin{proof} For $\alpha_1,\dots,\alpha_{\deg f}$ the roots of $f$ and $\beta_1,\dots, \beta_{\deg g}$ the roots of $g$,  \[ \Res(f,g) =\prod_{i=1}^{\deg f} \prod_{j=1}^{\deg g} (\beta_j - \alpha_i)\] and  \[ \Res(g,f) =\prod_{i=1}^{\deg f} \prod_{j=1}^{\deg g} (\beta_j - \alpha_i)\] so switching each term, we obtain $\deg f \deg g$ factors of $(-1)$. \end{proof}

Let $\Delta(f)$ be the discriminant of $f$. Let $\mu$ be the M\"{o}bius function. 

\begin{lemma}\label{mobius-evaluation} For $q$ odd we have \begin{equation}\label{me-1} \mu(f)=(-1)^{\deg f} \xi(\Delta(f)) \end{equation} and \begin{equation}\label{me-2} \Delta(f) = (-1)^{(\deg f) (\deg f-1)/2} \Res(f',f)\end{equation}  so \begin{equation}\label{me-3}\mu(f) =(-1)^{\deg f} (-1)^{\frac{ \deg f (\deg f-1)(q-1)}{4}} \left( \frac{f'}{f} \right)_{\xi} .\end{equation}\end{lemma}

\begin{proof} \eqref{me-1} is Pellet's formula. \eqref{me-2} follows from noting that for $\alpha_1,\dots,\alpha_{\deg f}$ the roots of $f$, we have $f'(\alpha_i) =\prod_{j \neq i} (\alpha_i- \alpha_j) $ so \[ \Res(f',f) = \prod_{1 \leq i \leq \deg f} \prod_{j \neq i} (\alpha_i-\alpha_j) = \prod_{1 \leq i < j \leq \deg f} (\alpha_i-\alpha_j) (\alpha_j-\alpha_i) \] \[= (-1)^{ \deg f(\deg f-1)/2}\prod_{1 \leq i < j \leq \deg f} (\alpha_i-\alpha_j)^2 =  (-1)^{ \deg f(\deg f-1)/2} \Delta(f) .\] \eqref{me-3} follows from combining \eqref{me-1}, \eqref{me-2}, and the fact that $\xi(-1) = (-1)^{ \frac{q-1}{2} } .$ \end{proof}

\begin{lemma}\label{gauss-evaluation}  If $q$ is odd then for $f_2$ squarefree and $f_1$ prime to $f_2$ we have
\begin{equation}\label{eq-gauss-evaluation} g_\chi(f_1,f_2) =   (-1)^{ \frac{ \deg f_2 (\deg f_2-1) (q-1)}{4}}  \left( \frac{f_2'}{f_2} \right)_\chi  \left( \frac{f_2'}{f_2} \right)_{\xi} \left( \frac{f_1}{f_2} \right)_\chi ^{-1}   (G(\chi,\psi))^{\deg f_2} .\end{equation}\end{lemma}

\begin{proof}  We first evalute the residue $\operatorname{res} \left(\frac{h f_1}{ f_2} \right)$ using the residue theorem. We define the residue of $\frac{hf_1}{f_2}$ at a root $\alpha$ of $f_2$ to be the coefficient of $\frac{1}{t-\alpha}$ in the Laurent series expansion of $\frac{ hf_1}{f_2}$ around $\alpha$.  The residue of $\frac{hf_1}{f_2}$ at $\infty$ is minus the coefficient of $1/t$ in the Laurent series expansion at $\infty$, i.e. $-\operatorname{res} (\frac{hf_1}{f_2})$. The residue theorem implies that the sum of the residues of $\frac{ hf_1}{f_2}$ at each point on the projective line vanishes. It follows that $\operatorname{res} (\frac{hf_1}{f_2})$ is the sum of the residues of $\frac{hf_1}{f_2}$ at the roots of $f_2$. For $\alpha$ a root of $f_2$, necessarily of order $1$, the residue of  $\frac{hf_1}{f_2}$  at $\alpha$ is the value of $\frac{ h f_1}{ f_2'}$ at $\alpha$. Summing these over $\alpha$ gives $\tr \frac{ hf_1}{f_2'}$ where $\tr \colon \mathbb F_q[t]/g \to \mathbb F_q$ is the trace. Thus
\[g_\chi(f_1,f_2)  = \sum_{h \in \mathbb F_q[t]/ f_2} \left( \frac{h}{f_2} \right)_\chi \psi \left( \operatorname{res} \left(\frac{h f_1}{ f_2} \right) \right) \]
\[ =\sum_{h \in \mathbb F_q[t]/ f_2} \left( \frac{h}{f_2} \right)_\chi \psi \left( \tr \frac{h f_1}{ f_2'} \right).\]

 If we change variables to $h^*=  f_1 h/f_2'$, we have $\left( \frac{h}{f_2} \right)_\chi  = \left( \frac{h^*}{f_2} \right)_\chi  \left( \frac{f_2'}{f_2} \right)_\chi  \left( \frac{f_1}{f_2} \right)_\chi ^{-1} $ so
\[g_\chi(f_1,f_2) = \left( \frac{f_2'}{f_2} \right)_\chi  \left( \frac{f_1}{f_2} \right)_\chi ^{-1} \sum_{h^* \in \mathbb F_q[t]/ f_2} \left( \frac{h^*}{f_2} \right)_\chi \psi \left( \tr h^* \right). \] 

The inner sum \[ \sum_{h^* \in \mathbb F_q[t]/ f_2} \left( \frac{h^*}{f_2} \right)_\chi \psi \left( \tr h^* \right)\] is multiplicative in $f_2$, and when $f_2$ is a prime $\pi$ takes the value $ - (-G(\chi,\psi))^{\deg \pi}$ by the Hasse-Davenport relations. Hence the inner sum is equal to $(-G(\chi,\psi))^{\deg f_2} \mu(f_2)$.  \eqref{eq-gauss-evaluation} then follows from the last identity of Lemma \ref{mobius-evaluation}.

\end{proof}

The term $\Res(f',f)$ that appears here has its own multiplicativity relation:

\begin{lemma}\label{discriminant-multiplicative} We have \[ \Res( (fg)' ,fg) = \Res(f',f) \Res(g',g) \Res(f,g) \Res(g,f) .\]\end{lemma}

\begin{proof}
\[ \Res( (fg)', fg)\] \[ = \Res ( (fg'+ f'g), f) \Res( (fg'+f'g),g)\] \[= \Res ( f'g, f) \Res( fg',g) \] \[= \Res (f',f) \Res(g,f) \Res(f,g) \Res(g',g). \qedhere\] \end{proof}

We record here also the multiplicativity relations for Gauss sums:

\begin{lemma}\label{gauss-multiplicative-easy} If $\gcd(f_2,f_3)=1$ then

\[ g_\chi(f_1f_3,f_2) = \left(\frac{f_3}{f_2} \right)_\chi^{-1}  g_\chi(f_1,f_2) .\]\end{lemma}

\begin{proof}\[ g_\chi(f_1f_3, f_2) = \sum_{h \in \mathbb F_q[t]/ f_2 } \left( \frac{h}{f_2} \right)_\chi \psi \left( \operatorname{res} \left(\frac{h f_1f_3}{ f_2} \right) \right).\]

Letting $h^* = hf_3$, we have \[ \left( \frac{h}{f_2} \right)_\chi  = \left( \frac{h^*}{f_2} \right)_\chi \left( \frac{f_3}{f_2} \right)_\chi ^{-1}, \]  and we observe that this change of variables is a permutation, so 
\[ \sum_{h \in \mathbb F_q[t]/ f_2 } \left( \frac{h}{f_2} \right)_\chi \psi \left( \operatorname{res} \left(\frac{h f_1f_3}{ f_2} \right) \right)= \sum_{h^* \in \mathbb F_q[t]/ f_2 } \left( \frac{h^* }{f_2} \right)_\chi \left( \frac{f_3}{f_2} \right)_\chi ^{-1} \psi \left( \operatorname{res} \left(\frac{h^*  f_1}{ f_2} \right) \right) = \left( \frac{f_3}{f_2} \right)_\chi ^{-1}  g_\chi(f_1,f_2).\]

\end{proof}

\begin{lemma}\label{gauss-multiplicative} If $\gcd(f_1,f_4) = \gcd(f_2,f_4)=\gcd(f_2,f_3)=1$ then \begin{equation}\label{eq-gauss-multiplicative} g_\chi(f_1f_3,f_2f_4) = g_\chi(f_1,f_2) g_\chi(f_3,f_4) \left(\frac{f_2}{f_4} \right)_\chi \left(\frac{f_4}{f_2} \right)_\chi  \left( \frac{f_1}{f_4} \right)_\chi^{-1} \left(\frac{f_3}{f_2} \right)_\chi^{-1} .\end{equation}\end{lemma}

\begin{proof} 
\[ g_\chi(f_1f_3, f_2f_4) = \sum_{h \in \mathbb F_q[t]/ (f_2f_4 )} \left( \frac{h}{f_2f_4 } \right)_\chi \psi \left( \operatorname{res} \left(\frac{h f_1f_3}{ f_2f_4} \right) \right).\]

As $f_2$ and $f_4$ are coprime, we can uniquely write $h= h_2 f_4 + h_4 f_2$ for $h_2  \in \mathbb F_q[t]/f_2$ and $h_4 \in \mathbb F_q[t] / f_4$. We then have  
\[  \left( \frac{h}{f_2f_4 } \right)_\chi = \left (\frac{h}{f_2}\right)_\chi  \left(\frac{h}{f_4} \right)_\chi = \left(\frac{h_2 f_4}{f_2} \right)_\chi \left( \frac{h_4f_2}{f_4} \right)_{\chi} = \left(\frac{h_2}{f_2} \right)_\chi \left(\frac{f_4}{f_2} \right)_\chi \left(\frac{h_4}{f_4} \right)_\chi \left(\frac{f_2}{f_4} \right)_\chi .\]

Furthermore, we have
\[ \psi\left( \operatorname{res} \left(\frac{h f_1f_3}{ f_2f_4} \right)\right) = \psi \left( \operatorname{res} \left(\frac{ h_2 f_1 f_3 }{ f_2} \right)\right)  \psi \left( \operatorname{res} \left(\frac{ h_4 f_1f_3}{ f_4} \right)\right) \] 

Hence
\[ g_\chi(f_1f_3, f_2f_4) =\left(\frac{f_4}{f_2} \right)_\chi \left(\frac{f_2}{f_4} \right)_\chi \left( \sum_{h_2 \in \mathbb F_q[t]/f_2}   \left(\frac{h_2}{f_2} \right)_\chi  \psi \left( \operatorname{res} \left(\frac{ h_2 f_1 f_3 }{ f_2} \right)\right)  \right)  \left( \sum_{h_4 \in \mathbb F_q[t]/f_4}  \left(\frac{h_4}{f_4} \right)_\chi \right) \]
\[= \left(\frac{f_4}{f_2} \right)_\chi \left(\frac{f_2}{f_4} \right)_\chi g_\chi(f_1f_3,f_2) g_\chi(f_1f_3,f_4).\]

Applying Lemma \ref{gauss-multiplicative-easy} to each factor, we get \eqref{eq-gauss-multiplicative}.\end{proof}

Another identity to evaluate Gauss sums will help compare with the work of Chinta and Mohler.

\begin{lemma}\label{gauss-prime-powers} For $\chi$ of order $n$ and $\pi$ prime, we have \[ g_\chi(\pi^{d_1}, \pi^{d_2} ) =  \begin{cases} 1 & \textrm{ if } d_2=0 \\
(q^{\deg \pi} - 1)  q^{ (d_2 -1 ) \deg \pi} & \textrm{ if }d_2 \equiv 0 \bmod \nchi\textrm{ and }d_1 \geq d_2 \\
0  & \textrm{ if }d_2 \not\equiv 0 \bmod \nchi\textrm{ and }d_1 \geq d_2 \\
- q^{(d_2-1 ) \deg \pi}  \left(\frac{\pi'}{\pi}\right)^{d_2}_\chi  (- G(\chi^{d_2} ,\psi))^{\deg \pi} & \textrm{ if } d_1= d_2 -1 \\
0 & \textrm{if } d_1< d_2 -1 \end{cases} .\]\end{lemma}

\begin{proof} We begin by noting
\begin{equation} \label{eq-gpp} g_\chi ( \pi^{d_1} ,\pi^{d_2} ) = \sum_{h \in \mathbb F_q[t]/ \pi^{d_2} } \left( \frac{h}{\pi^{d_2} } \right)_\chi \psi \left( \operatorname{res} \left(h \pi^{ d_1-d_2}  \right) \right).\end{equation}

First, if $d_2=0$, the sum \eqref{eq-gpp} has a single term and equals $1$. Second, $\left( \frac{h}{\pi^{d_2} } \right)_\chi$ depends only on $h \bmod \pi$, so if $d_1 < d_2-1$, the $\psi$ term cancels in each residue class mod $\pi$ and so the sum \eqref{eq-gpp} vanishes. If $d_1 \geq d_2$, the $\psi$ term can be ignored and the sum \eqref{eq-gpp} vanishes because the multiplicative character cancels, unless $d_2\equiv 0 \bmod \nchi$, in which case the summand is $1$ if $h$ is prime to $\pi$ and $0$ otherwise, and the value of the sum \eqref{eq-gpp} is simply the number of $h$ prime to $\pi$, which is $(q^{\deg \pi} -1 ) q^{ (d_2 -1 ) \deg \pi}$. 

If $d_1 =d_2-1$, the sum \eqref{eq-gpp} is equal to \[q^{(d_2-1 ) \deg \pi} \sum_{h \in \mathbb F_q[t]/ \pi} \left( \frac{h}{\pi } \right)^{d_2}_\chi \psi \left( \operatorname{res} \left(\frac{h}{\pi}  \right) \right) = q^{(d_2-1)\deg \pi } g_{ \chi^{d_2}} (1,\pi) \]
\[=q^{(d_2-1 ) \deg \pi} (-1)^{ \frac{ \deg \pi (\deg \pi-1) (q-1)}{4}}  \left( \frac{\pi'}{\pi} \right)_\chi^{d_2} \left( \frac{\pi'}{\pi} \right)_{\xi}  (G(\chi^{d_2},\psi))^{\deg \pi}  =- q^{(d_2-1 ) \deg \pi}  \left(\frac{\pi'}{\pi}\right)^{d_2}_\chi  (- G(\chi^{d_2},\psi))^{\deg \pi} \]
by Lemma \ref{gauss-evaluation} and \eqref{me-3}, verifying the last remaining case. \end{proof}


\subsection{$\ell$-adic sheaves}
We have the following basic properties of $IC_{\mathcal L_\chi(f)}$.

\begin{lemma}\label{basic-IC}

\begin{enumerate}

\item For $f$ a function on $X$ and $g$ an invertible function on $X$, \[IC_{\mathcal L_\chi(fg)}\cong  IC_{\mathcal L_\chi(f)} \otimes \mathcal L_\chi(g).\]

\item For $s\colon X\to Y$ a smooth map and $f$ a function on $Y$, \[IC_{\mathcal L_\chi(f \circ s)} \cong s^* IC_{\mathcal L_\chi(f)}.\]

\item For $X$ and $Y$ two varieties, $f$ a function on $X$ and $g$ a function on $Y$, \[IC_{\mathcal L_\chi ( (x,y) \mapsto f(x)g(y))} \cong IC_{\mathcal L_\chi(f) } \boxtimes IC_{\mathcal L_\chi(g)}.\]

\item For $f$ a function on $X$, with $X$ of dimension $d$, and $D$ the Verdier dual,

\[ D IC_{\mathcal L_\chi(f)} \cong IC_{\mathcal L_{\chi^{-1}} (f)} [2d] (d).\]

\end{enumerate}

\end{lemma}

\begin{proof} These all are proved by combining a basic property of middle extension with a property of the sheaves $\mathcal L_\chi$ that follows in a straightforward way from their definition. 

(1) follows from the fact that middle extension is compatible with tensor product with lisse sheaves, and the fact that $\mathcal L_\chi(f) \otimes \mathcal L_\chi(g) = \mathcal L_\chi(fg)$.

(2) follows from the fact that middle extension is compatible with smooth pullback (once shifts are taken into account) and $s^* \mathcal L_\chi(f) = \mathcal L_\chi(f \circ s)$.

(3) follows from the fact that both middle extension and $\mathcal L_\chi$ are compatible with $\boxtimes$.

(4) follows from the fact that middle extension is compatible with Verdier duality and $\mathcal L_\chi$ is dual to $\mathcal L_{\chi^{-1}}$ as a lisse sheaf, hence $D \mathcal L_\chi(f) = \mathcal L_{\chi^{-1}} (f) [2d](d).$ 

These middle extension compatibilities follow from the, even more standard, compatibilities of $j_!$ and $j_*$ with these operations.\end{proof}

We need also a slightly more complicated observation along the same lines. First, we define and describe the notion of the Weil restriction of a complex of sheaves, building on the notion of a tensor direct image of sheaves defined by \cite{ARL}.

\begin{defi} Let $k'/k$ be a finite Galois field extension. Let $X$ be a variety over $k'$. The \emph{Weil restriction}  $WR_{k'}^k X$ is defined as the variety over $k$ whose $R$ points for a $k$-algebra $R$ are the $R\otimes_k k'$-points of $X$.

For $R$ a $k'$-algebra, the natural map $R\otimes_k k' \to R$ defines a map from $R$-points of $WR_{k'}^k X$ to $R$-points of $X$, defining a map $\rho \colon  (WR_{k'}^k X)_{k'} \to X$.

Let $\pi \colon (WR_{k'}^k X)_{k'}\to WR_{k'}^k X$ be the natural map.

For $K'$ a complex on $(WR_{k'}^k X)_{k'}$, \citet[Definition 2 on p. 133]{ARL} defines the \emph{tensor direct image}  $\pi_{\otimes *} K'$ as the unique complex on $WR_{k'}^k X$ whose pullback to $(WR_{k'}^k X)_{k'}$ is isomorphic to $\bigotimes_{\tau\in \Gal(k'/k) }\tau^* K'$ where the natural action of $\Gal(k'/k) $ on the pullback is equal to the natural action of $\Gal(k'/k) $ permuting the factors (which exists and is unique by \citep[Proposition 8 on p. 133]{ARL}). 

For $K$ a complex on $X_{k'}$, define the \emph{Weil restriction}  $WR_{k'}^k K$ by
\[ WR_{k'}^k K = \pi_{\otimes *} \rho^* K.\] \end{defi}

\begin{remark} Note that this definition uses complexes of sheaves rather than the derived category of sheaves because the descent argument needed to prove existence and uniqueness would, in the derived category, require checking higher compatibilities of the action. If $K$ is an ordinary sheaf, or a perverse sheaf, up to shift, these subtleties can be avoided, as these categories satisfy \'{e}tale descent. We will only apply this in the case of perverse sheaves up to shift. \end{remark}


%
%

\begin{lemma} \label{weil-restriction-trace-function} Let $X$ be a variety over $\mathbb F_{q^d}$. Let $K$ be a perverse sheaf on $X$. Then the trace of Frobenius on the stalk of $WR_{\mathbb F_{q^d}}^{\mathbb F_q} K $ at an $\mathbb F_q$-point is equal to the trace of Frobenius on the stalk of  $K$ at the corresponding $\mathbb F_{q^d}$-point, using the natural bijection $X(\mathbb F_{q^d}) = WR_{\mathbb F_{q^d}}^{\mathbb F_q} X ( \mathbb F_q)$. \end{lemma}

\begin{proof} By definition and \cite[Proposition 9 on p. 133]{ARL}, the trace of $\operatorname{Frob}_q$ on the stalk of $WR_{\mathbb F_{q^d}}^{\mathbb F_q} K$ at an $\mathbb F_q$-point $x$ is the trace of $\operatorname{Frob}_{q^d}$ on the stalk of $\rho^* K$ on the $\mathbb F_{q^d}$-point $\pi^{-1}(x)$. The stalk of $\rho^* K$ at $\pi^{-1}(x)$ is the stalk of $K$ at $\rho(\pi^{-1}(x))$, which is the corresponding $\mathbb F_{q^d}$-point of $X$. \end{proof}

\begin{lemma}\label{norms} Let $k'/k$ be a finite Galois field extension of fields containing $\mu_{n}$. Let $X$ be a variety over $k'$ and $f$ a function on $X$. Let $WR_{k'}^k X$ be the Weil restriction from $k'$ to $k$ of $X$. The function $f$ on $X$ induces a map $ WR_{k'}^k X \to WR_{k'}^k \mathbb A^1$, which we can compose with the norm map $WR_{k'}^k \mathbb A^1 \to \mathbb A^1$ to obtain a function $Nf$ on $WR_{k'}^k X$. Let $WR_{k'}^k IC_{\mathcal L_\chi(f)}$ be the Weil restriction of $IC_{\mathcal L_\chi(f)}$. 
Then
\begin{equation} \label{eq-norms}  WR_{k'}^k IC_{\mathcal L_\chi(f)} \cong  IC_{\mathcal L_\chi( N f) }.\end{equation} \end{lemma}

\begin{proof}  Since $k'/k$ is Galois, we have an isomorphism \[ (WR_{k'}^k X)_{k'} = \prod_{ \tau \in \Gal(k'/k)} X \] with the projection onto the $\tau$'th factor given by $\rho\circ \tau$.

By definition the pullback of $WR_{k'}^k IC_{\mathcal L_\chi(f)} $ to $k'$ is given by
\[ \bigotimes_{ \tau \in \Gal(k'/k)}  \tau^* \rho^* IC_{\mathcal L_\chi(f)} = \boxtimes_{ \tau \in \Gal(k'/k) }  IC_{\mathcal L_\chi(f)} = IC_{ \mathcal L_\chi( \prod_{ \tau \in \Gal(k'/k)} f\circ \rho \circ \tau )} = IC_{ \mathcal L_\chi(N f) } \] by Lemma \ref{basic-IC}(3) and the identity $\prod_{ \tau \in \Gal(k'/k)} f\circ \rho \circ \tau = Nf$ on  $(WR_{k'}^k X)_{k'}$. So the two complexes in \eqref{eq-norms} are isomorphic after pullback to $k'$.

Since $IC_{ \mathcal L_\chi(N f) }$ is the middle extension of a lisse sheaf of rank one, it follows that $WR_{k'}^k IC_{\mathcal L_\chi(f)}  $ is the middle extension of a lisse sheaf of rank one as well. To check they are isomorphic over $k$, it suffices to check the lisse sheaves are isomorphic, for which, because they are isomorphic over $k'$, it suffices to check that their stalks at a single point are isomorphic as Galois representations.

 For $WR_{k'}^k IC_{\mathcal L_\chi(f)}$, the stalk at a geometric point  $x \in WR_{k'}^k X$ where $Nf$ is nonzero is naturally the tensor product of a one-dimensional vector space for each $\tau \in \Gal(k'/k)$, and on each one-dimensional vector space the action is the same as on an $\nchi$'th root of $f(\rho(\tau(x))$. For $IC_{ \mathcal L_\chi(N f) }$, the Galois action is the same as the Galois action on the $\nchi$th root of $Nf (x)$. Because $Nf(x) = \prod_{\tau\in \Gal(k'/k)} f(\rho(\tau(x))$, and the $\nchi$th root of the product is the product of the $\nchi$th roots of the factors, these are the same.\end{proof}

\begin{lemma}\label{scale-invariance} Let $X$ be a variety with an action of $\mathbb G_m$ described by a map $a\colon X \times \mathbb G_m \to X$. Let $f$ be a function on $X$ and $r$ an integer such that $f(a(x,\lambda))= f(x)\lambda^r$ for all $x\in X$ and $\lambda\in \mathbb G_m$.

If $r$ is divisible by $\nchi$, then $IC_{\mathcal L_\chi(f)}$ is $\mathbb G_m$-invariant, in the sense that $a^* IC_{\mathcal L_\chi(f)} = IC_{\mathcal L_{\chi}(f) }\boxtimes \mathbb Q_\ell$. In particular, this always happens if we compose $a$ with the $\nchi$'th power homomorphism $\mathbb G_m \to \mathbb G_m$.

If $r$ is not divisible by $\nchi$, then the stalk of $IC_{\mathcal L_\chi(f)}$ vanishes at every $\mathbb G_m$-invariant point.\end{lemma}

\begin{proof} Because $a$ is smooth, we have by Lemma \ref{basic-IC}(2,3) \[ a^* IC_{\mathcal L_\chi(f)}  \cong IC_{\mathcal L_\chi ( f \circ a)}\cong IC_{\mathcal L_\chi( f(x) \lambda^r) }\cong IC_{ \mathcal L_\chi(f)} \boxtimes IC_{\mathcal L_\chi(\lambda^r)} .\]

If $r$ is divisible by $\nchi$, then $IC_{\mathcal L_\chi( \lambda^r)}$ is the middle extension of the constant sheaf, hence is simply the constant sheaf.

If $r$ is not divisible by the order of $\chi$, then restricting this identity to $ P \times \mathbb G_m$ for a $\mathbb G_m$-fixed point $P$, we have $(IC_{\mathcal L_\chi(f)})_P \otimes \mathbb Q_\ell =  (IC_{\mathcal L_\chi(f)})_P \otimes \mathcal L_\chi(\lambda^r)$. Because one side has trivial monodromy and the other nontrivial, they cannot be isomorphic unless they both vanish. \end{proof}

\begin{lemma}\label{equivariant-vanishing} Let $B$ be a scheme of finite type over a field, $Y = B \times \mathbb A^1$, $u\colon B \times \mathbb G_m \to B \times \mathbb A^1$ the inclusion, $\pi\colon B \times \mathbb A^1 \to B$ the projection, $K$ a complex on $B \times \mathbb G_m$, and $N\neq 0$ an integer.

Assume that $K$ is invariant for the action of $\mathbb G_m$ on $B \times \mathbb G_m$  given by $a ( (b,\lambda_1), \lambda_2) = (b,\lambda_1 \lambda_2^N)$ for all $b\in B, \lambda_1,\lambda_2\in \mathbb G_m$.

Then $ \pi_* u_! K =0$. \end{lemma}

\begin{proof} Let $\overline{\pi} \colon \mathbb A^1 \to \operatorname{pt}$ be the projection and $\overline{u} \colon \mathbb G_m \to \mathbb A^1$ the inclusion, so that $u = id \times \overline{u}$ and $\pi = id \times \overline{\pi}$. Let $\rho \colon \mathbb G_m \to \mathbb G_m$ be the $N$th power map. Let $i \colon \operatorname{pt} \to \mathbb G_m$ be the inclusion of the identity.

Restricting the $\mathbb G_m$-invariance property to the locus where $\lambda=1$	, we see that $ (id \times \rho)^* K  = ((id \times i)^* K )\boxtimes \mathbb Q_\ell $. Since $\rho$ is finite, it follows that $K$ is a summand of $(id\times\rho)_* ((id \times i)^* K )\boxtimes \mathbb Q_\ell$, so it suffices to prove the vanishing of
\[ \pi_* u_! (id\times\rho)_*(  ((id \times i)^* K )\boxtimes \mathbb Q_\ell)=  (id \times \overline{\pi})_*   (id \times \overline{u})_! (id\times\rho)_* (((id \times i)^* K )\boxtimes \mathbb Q_\ell). \]
But by the K\"unneth formula in the form \cite[Corollary 9.3.5]{leifu}, together with its compactly supported version \cite[Corollary 7.4.9]{leifu}, we have 
\[ (id \times \overline{\pi})_*   (id \times \overline{u})_! (id\times\rho)_* (((id \times i)^* K )\boxtimes \mathbb Q_\ell) = (id \times \overline{\pi})_*   (id \times \overline{u})_! (((id \times i)^* K )\boxtimes  \rho_* \mathbb Q_\ell)\]
\[= (id \times \overline{\pi})_*   (((id \times i)^* K )\boxtimes  \overline{u}_!  \rho_* \mathbb Q_\ell)=((id \times i)^* K )\boxtimes  \overline{\pi}_* \overline{u}_!  \rho_* \mathbb Q_\ell.\]
Thus it suffices to prove \[\overline{\pi}_* \overline{u}_!  \rho_* \mathbb Q_\ell=0,\] but
$\overline{\pi}_* \overline{u}_!  \rho_* \mathbb Q_\ell$ is a complex on a point, given by the cohomology groups $H^* (\mathbb A^1, \overline{u}_! \rho_* \mathbb Q_\ell)$.

By Artin's theorem \cite[Corollary 7.5.2]{leifu}, this cohomology vanishes in all degrees but zero and one. All global sections of $\overline{u}_! \rho_* \mathbb Q_\ell$ vanish at zero, hence vanish in a neighborhood of zero, hence vanish everywhere because $\overline{u}_! \rho_* \mathbb Q_\ell$ is lisse away from zero, so $H^0$ vanishes. By the Grothendieck-Ogg-Shafarevich Euler characteristic formula \cite[Theorem 7.1]{sga5}, the Euler characteristic of $\overline{u}_! \rho_* \mathbb Q_\ell$ is zero, so $H^1$ vanishes as well. \end{proof}

It is a classical fact that the ring of functions on an affine scheme with a $\mathbb G_m$-action is a $\mathbb Z$-graded ring (and the graded structure is equivalent to the $\mathbb G_m$ action). Indeed, for $R$ the ring of functions, the action of $\mathbb G_m$ defines a ring homomorphism $a^* \colon R \to R[ \lambda, \lambda^{-1} ] $ and one can define $R_d$ for $d\in \mathbb Z$ as the set of $x$ with $a^*(x)=x \lambda^d$. Then if we write $ a^*(x) = \sum_{d\in \mathbb Z} x_d \lambda^d$ with all but finitely many of the $x_d$ zero, associativity of the action implies $x_d \in R_d$ and identity implies $x = \sum_d x_d$ so that $R = \bigoplus_d R_d$, and the fact that $a^*$ is a ring homomorphism implies that $R_{d_1} \cdot R_{d_2} \subseteq R_{d_1+d_2}$.

\begin{lemma}\label{cone-contractible} Let $X$ be an affine scheme of finite type over a field with a $\mathbb G_m$-action. Equivalently, the ring of functions on $X$ is a graded ring.

Assume that all nonconstant homogeneous functions on $X$ have positive degree. Let $P$ be the unique $\mathbb G_m$-fixed point of $X$ (corresponding to the ideal generated by homogenous functions of positive degree). Let $K$ be a $\mathbb G_m$-invariant complex on $X$. Then \[ H^*(X, K) \cong K_{P}.\] \end{lemma}

\begin{proof} Let $x_1,\dots,x_n$ be generators of the ring of functions on $X$ of degrees $d_1,\dots,d_n$. Let $d$ be the least common multiple of $d_1,\dots,d_n$. Then $(x_1^{d/d_1},\dots, x_n^{d/d_n})$ defines a finite $\mathbb G_m$-equivariant map from $X$ to $\mathbb A^n$, where $\mathbb G_m$ acts on $\mathbb A^n$ by multiplying all coordinates by the $d$th power. Because the map is finite, and $P$ is the unique point in the inverse image of $0 \in \mathbb A^n$, both $H^*(X,K)$ and $K_p$ are preserved by pushing forward along this map, and because this map is $\mathbb G_m$-equivariant, the $\mathbb G_m$-invariance is preserved. So we can reduce to the case where $X = \mathbb A^n$.

Let $j$ be the inclusion from $\mathbb A^n -\{0\}$ to $\mathbb A^n$. From the excision exact sequence $j_! j^* K \to K \to K_p$, it suffices to prove $H^*(\mathbb A^n, j_! j^* K)=0$. Let $Y$ be the blowup of $\mathbb A^n$ at the origin, let $u\colon \mathbb A^n-\{0\}\to Y$ be the inclusion, $b\colon Y\to \mathbb A^n$ the blowup map, and $\pi\colon Y \to \mathbb P^{n-1}$ the projection onto the exceptional fiber. We have $j = b \circ u$ and $b$ is proper so 
\[ H^*(\mathbb A^n, j_! j^* K)= H^*(\mathbb A^n, b_! u_! j^* K)=  H^*(\mathbb A^n, b_* u_! j^* K)=  H^*(Y, u_! j^* K)= H^* (\mathbb P^{n-1}, \pi_* u_! j^* K).\]

Thus it suffices to show that $\pi_* u_!  K'$ is zero for a $\mathbb G_m$-equivariant sheaf $K'$ on $\mathbb A^n-\{0\}$. Locally on $\mathbb P^{n-1}$, $Y$ is an $\mathbb A^1$-bundle, $\pi$ the structure map, $u$ the inclusion of the complement of the $0$ section, and the $\mathbb G_m$ action is by multiplication by the $d$th power. To prove this vanishing, we work locally on $\mathbb P^{n-1}$, where we are in the setting of Lemma \ref{equivariant-vanishing}. We take $B$ to be an open subset of $\mathbb P^{n-1}$ where this bundle can be trivialized and let $K$ be the pullback of $K'$ along this trivialization.  By Lemma \ref{equivariant-vanishing}, $\pi_* u_! K'=0$.\end{proof}

\section{Proofs of the Axioms}\label{axioms}

We are now ready to check that the function $a$ satisfies the axioms of Theorem \ref{axiomatics}.

\begin{lemma}\label{axiom-i} If $f_1,\dots,f_{r}$ and $g_1,\dots, g_{r}$ satisfy $\gcd(f_i,g_j)=1$ for all $i$ and $j$, then we have \[ a ( f_1g_1,\dots, f_{r} g_{r}; q , \chi, M) \] \[= a(f_1,\dots,f_{r} ; q, \chi, M) a(g_1,\dots, g_{r} ; q, \chi M)  \prod_{1 \leq i \leq r} \left( \frac{f_i}{g_i} \right)_{\chi}^{M_{ii}} \left( \frac{g_i}{f_i} \right)_{\chi}^{M_{ii}}   \prod_{1 \leq i < j \leq r}  \left( \frac{f_i}{g_j} \right)_{\chi}^{M_{ij}}\left( \frac{g_i}{f_j} \right)_{\chi}^{M_{ij}} .\] \end{lemma}

\begin{proof} Let $d_i= \deg f_i$ and $e_i = \deg g_i$. Consider the map $\mu\colon \prod_{i=1}^{r} \mathbb A^{d_i} \times \prod_{i=1}^{r} \mathbb A^{e_i}\to \prod_{i=1}^{r} \mathbb A^{d_i+e_i} $  defined by polynomial multiplication. 

Observe that \[\Res(f_i g_i, f_j g_j) = \Res(f_i,f_j) \Res(g_i,g_j) \Res(g_i ,f_j) \Res(f_i,g_j)\] and by Lemma \ref{discriminant-multiplicative} \[ \Res ((f_ig_i)' , f_i g_i)= \Res(f_i',f_i) \Res(g_i',g_i) \Res(f_i,g_i ) \Res(g_i,f_i). \] 
Hence, letting \[ G= \prod_{1 \leq i \leq r} \Res(f_i, g_i)^{M_{ii}} \Res(g_i,f_i)^{M_{ii} }\prod_{1 \leq i < j \leq r} \Res(f_i, g_j)^{M_{ij}}  \Res(g_i,f_j)^{M_{ij}}  \] we have
 \begin{equation}\label{axiom-i-key-eq}   F(f_1g_1,\dots, f_r g_r)  =F(f_1, \dots,f_r) F(g_1,\dots,g_r) G  .\end{equation}

Let $U$ be the open set $U \subseteq \prod_{i=1}^r \mathbb A^{d_i} \times \prod_{i=1}^r \mathbb A^{e_i}$ where $\gcd(f_i,g_j)=1$ for all $i$ and $j$. Certainly, $G$ has no zeroes or poles on $U$ as the resultants are all nonvanishing on $U$.

Let us check that $\mu$ is \'{e}tale on $U$. The derivative of $\mu$ at a point $f_1,\dots, f_r, g_1,\dots, g_r$ is the linear map that sends a tangent vector $df_1,\dots, df_r, dg_1,\dots, dg_r$, where $df_i$ is a polynomial of degree $<d_i$ and $dg_i$ is a polynomial of degree $<e_i$, to the vector $ g_1 df_1 + f_1 dg_1,\dots, g_r df_r + f_r dg_r$. The derivative is injective unless there exist  polynomials $df_1,\dots, df_r, dg_1,\dots, dg_r$, not all zero, such that $g_i df_i =- f_i dg_i$ for all $i$. If $f_i$ and $g_i$ do not share a common root, this equation implies that $f_i$ divides $df_i$ which, since $\deg df_i < d_i = \deg f_i$, implies that $df_i=0$, and similarly with $dg_i$. Thus the derivative is injective on $U$. Since, restricted to $U$, $\mu$ is a map between smooth varieties of the same dimension with injective derivative, it is \'{e}tale. 

In particular, $\mu$ is smooth on $U$.  Hence by applying Lemma \ref{basic-IC}(1,2,3) to \eqref{axiom-i-key-eq}, we obtain an isomorphism  \[ \mu^* K_{d_1+e_1,\dots,d_{r} + e_{r}} \cong\left(  K_{d_1,\dots,d_{r}} \boxtimes K_{e_1,\dots,e_{r}} \right) \otimes  \mathcal L_\chi \left(G  \right)  \]  on $U$. Taking trace functions of both sides, and applying Lemma \ref{resultant} to evaluate the trace function $\chi(G)$ of $\mathcal L_\chi(G)$, we get the stated identity.\end{proof}
 
 \begin{lemma}\label{axiom-ii} $a(1,\dots, 1, f, 1,\dots,1;q,\chi,M)=1$ for all linear polynomials $f$. \end{lemma}
 
 \begin{proof} In this case, all resultants and discriminants are $1$, so $F=1$, thus $ K_{0,\dots, 0, 1, 0,\dots, 0}$ is the constant sheaf $\Ql$, hence its trace function is $1$. \end{proof}
 
To check the remaining axioms, it will be useful to describe the translation and dilation symmetries of the function $F$.

 \begin{lemma}\label{dilation-invariance}
 
 \begin{enumerate}
 
 \item We have \[ F( \lambda^{d_1} f_1(x/\lambda),\dots, \lambda^{d_r} f_r(x/\lambda)) =\lambda^{ \sum_{i=1}^r d_i (d_i-1) M_{ii} + \sum_{1 \leq i < j \leq r} d_i d_j M_{ij} } F( f_1,\dots, f_r) . \]
 
 \item All nonconstant polynomials on $\prod_{i=1}^{r} \mathbb A^{d_i} $ which are homogeneous for the action of $\mathbb G_m$ on $\prod_{i=1}^{r} \mathbb A^{d_i} $  which acts by dilation of polynomials, i.e. $f_i \to \lambda^{d_i} f_i (x/\lambda)$, have positive degree in $\lambda$.
 
 \end{enumerate}
 
 \end{lemma}
 
 \begin{proof} (1) follows from the definition of the resultant of two monic polynomials as a product of differences of their roots, since dilation multiplies each root by $\lambda$, hence each difference of roots by $\lambda$, and thus a product of $N$ differences of roots by $\lambda^N$.
 
 (2) holds because the ring of functions is generated by the coefficients of $f_i$, which all have positive degree in $\lambda$.\end{proof}
 
 \begin{lemma}\label{translation-invariance} The complex $K_{d_1,\dots, d_r}$ is invariant under the action of $\mathbb G_a$ on $\prod_{i=1}^r \mathbb A^{d_i}$ given by $( (f_1,\dots, f_r), \alpha)  \mapsto  ( f_1(T+\alpha),\dots, f_r(T+\alpha))$. \end{lemma}
 
 \begin{proof} This follow from Lemma \ref{basic-IC} and the identity \[ F_{d_1,\dots, d_r}  ( f_1(T+\alpha),\dots, f_r(T+\alpha))= F_{d_1,\dots, d_r}( f_1,\dots, f_r) \] which is immediate from the definition of $F_{d_1,\dots,d_r}$.\end{proof}


 For each finite field $\mathbb F_q$, character $\chi$, and natural numbers $d_1,\dots, d_r$, let $J(d_1,\dots, d_r; q, \chi, M)$ be the finite set of ordered pairs of Weil numbers and integers given by the eigenvalues of $\Frob_q$ on the stalk of $K_{d_1,\dots, d_r}$ at $(T^{d_1},\dots, T^{d_r})$, together with their signed multiplicities.
 
 \begin{lemma}\label{compatibility} For each tuple of natural numbers $d_1,\dots, d_r$, the function $q,\chi \mapsto J(d_1,\dots, d_r; q, \chi, M)$ is a compatible system of sets of ordered pairs of Weil numbers and integers. \end{lemma}
 
 \begin{proof}   Let $\mathbb F_q$ be a finite field, $\chi$ a character, $\mathbb F_{q^m}$ a field extension, and $\chi_m$ the composition of $\chi$ with the norm map. We can construct  $K_{d_1,\dots, d_r} $ on $\prod_{i=1}^r \mathbb A^{d_i}_{\mathbb F_q}$ using the character $\chi$ and then pull back to $\prod_{i=1}^r \mathbb A^{d_i}_{\mathbb F_{q^m}}$, or we can construct the sheaf directly on $\prod_{i=1}^r \mathbb A^{d_i}_{\mathbb F_{q^m}}$  using the character $\chi_m$.  These two sheaves are naturally isomorphic because the $\frac{q^m-1}{q-1}$th-power map $\mu_{q^m-1} \to \mu_{q-1}$ used to compare the Kummer sheaves matches the norm map $\mathbb F_{q^m}^\times \to \mathbb F_q^\times$ given by \[N(x) = x \cdot \Frob_q(x) \cdot \ldots \cdot \Frob_q^{m-1}(x) =x \cdot x^q \cdot \ldots \cdot x^{q^{m-1}} = x^{  \frac{q^m-1}{q-1}} \] used to convert $\chi$ to $\chi_m$, and because forming the intermediate extension commutes with change of base field. 
 
 The elements of $J(d_1,\dots, d_r; q^m, \chi_m, M)$ are the eigenvalues of $\operatorname{Frob}_{q^m}$, with multiplicity, on the stalk of $K_{d_1,\dots, d_r}$, constructed over $\mathbb F_{q^m}$ using $\chi_m$, and hence also the eigenvalues of $\operatorname{Frob}_{q^m}$ on the stalk of $K_{d_1,\dots,d_r}$, constructed over $\mathbb F_q$ and pulled back to $\mathbb F_{q^m}$. Since $\operatorname{Frob}_{q^m} = \operatorname{Frob}_q^m$, the elements of $J(d_1,\dots, d_r; q^m, \chi_m, M)$ are the $m$'th powers, with multiplicity, of the eigenvalues of $\operatorname{Frob}_q$ on the stalk of $K_{d_1,\dots,d_r}$, i.e. the $m$'th powers of the elements of $J(d_1,\dots, d_r; q,\chi, M)$, as desired. \end{proof}

 \begin{lemma}\label{axiom-iii}  For each finite field $\mathbb F_q$, character $\chi$, natural numbers $d_1,\dots, d_r$, and prime polynomial $\pi$ over $\mathbb F_q$, we have
    \begin{equation}\label{eq-axiom-iii} a(\pi^{d_1},\dots, \pi^{d_r}; q, \chi, M) =\left(\frac{\pi'}{\pi}\right)_{\chi} ^{ \sum_{i=1}^r { d_i M_{ii}} }   \sum_{j \in J(d_1,\dots,d_r; q,\chi, M)}  c_j \alpha_j^{\deg \pi}  .\end{equation} \end{lemma}
    
   \begin{proof}This follows from the definition when $\pi=T$, and then follows from Lemma \ref{translation-invariance} when $\pi= T-x$ for $x\in \mathbb F_q$.

Let us handle the case when $\pi$ has a higher degree. To do this, let $e$ be the degree of $\pi$, and consider the Weil restriction $WR_{\mathbb F_{q^e}}^{\mathbb F_q} \prod_{i=1}^{r} \mathbb A^{d_i}$ of $\prod_{i=1}^{r} \mathbb A^{d_i}$ from $\mathbb F_{q^e}$ to $\mathbb F_{q}$. This Weil restriction admits a map $norm$ to $\prod_{i=1}^{r} \mathbb A^{ed_i}$ given by taking norms of polynomials. For $x$ a root of $\pi$, the image of $((T-x)^{d_1},\dots,(T-x)^{d_{r}})$ under $norm$ is $(\pi^{d_1},\dots, \pi^{d_{r}})$. Thus
\begin{equation}\label{axiom-iii-left} \begin{split}& a(\pi^{d_1},\dots, \pi^{d_r}; q, \chi, M)  = \operatorname{tr} \left(\Frob_q,  (K_{en_1,\dots, en_{r}})_{ (\pi^{d_1},\dots, \pi^{d_r})} \right) \\ & =  \operatorname{tr} \left(\Frob_q,  (norm^* K_{en_1,\dots, en_{r}})_{ ((T-x) ^{d_1},\dots, (T-x) ^{d_r})} \right).\end{split}  \end{equation}

On the other hand, by Lemma \ref{weil-restriction-trace-function},
\begin{equation} \label{axiom-iii-right} \begin{split} & \operatorname{tr}  \left(( \Frob_q,   ( WR_{\mathbb F_{q^e}}^{\mathbb F_q} K_{d_1,\dots,d_{r}} )_{ ((T-x)^{d_1},\dots,(T-x)^{d_{r}})}\right) \\ & = \operatorname{tr} \left( \Frob_{q^e}, (K_{d_1,\dots, d_{r}})_{ ((T-x)^{d_1},\dots,(T-x)^{d_{r}})} \right)= \sum_{j \in J(d_1,\dots,d_r; q,\chi, M)}  c_j \alpha_j^{\deg \pi}. \end{split}
\end{equation} 

To finish the argument, we will compare the stalks of  $WR_{\mathbb F_{q^e}}^{\mathbb F_q}K_{d_1,\dots,d_{r}}$ and $norm^* K_{d_1e,\dots, d_re}$ at $((T-x)^{d_1},\dots,(T-x)^{d_{r}})$. To do this, note from Lemma \ref{norms} that \begin{equation}\label{axiom-iii-WR-iC} WR_{\mathbb F_{q^e}}^{\mathbb F_q}K_{d_1,\dots,d_{r}} \cong IC_{\mathcal L_\chi (NF (f_1,\dots, f_r))}.\end{equation}

The restriction of $norm$ to the open set where none of the polynomials share any roots with their Galois conjugates is \'{e}tale, so \[norm^* IC_{\mathcal L_\chi ( F_{d_1e,\dots,d_{r}e})} = IC_{\mathcal L_\chi (F_{d_1e,\dots,d_{r}e} \circ norm)}\]  by Lemma \ref{basic-IC}(2).  Let $N$ be the norm map on polynomials. By definition, we have identities
  \begin{equation} \label{norm-ratio-numerator} F(N f_1,\dots, N f_r)  = \prod_{1 \leq i \leq r} \left(   \Res ( (N f_i)', Nf_i)  \right)^{M_{ii}}  \prod_{1 \leq i<j\leq r} \left( \Res (N f_i, N f_j)\right)^{M_{ij}} \end{equation}
  \begin{equation} \label{norm-ratio-denominator} N F(f_1,\dots, f_r) = \prod_{1 \leq i \leq r} \left(  N \Res (f_i',f_i) \right)^{M_{ii}}  \prod_{1 \leq i<j\leq r} \left( N \Res (f_i, f_j) \right)^{M_{ij}} .\end{equation}

For compactness, we will here write $\FFrob_q$ for $\Frob_q$. The multiplicativity of the resultant gives \begin{equation}\label{resultant-norm-1} \Res (N f_i, N f_j)  =N \Res (f_i, f_j)  \prod_{\substack{ 0 \leq t_1, t_2 \leq e-1 \\ t_1 \neq t_2} }  \Res( \FFrob_q^{t_1} f_i, \FFrob_q^{t_2} f_j) \end{equation} and the multiplicativity property of $\Res(f',f)$ gives  \begin{equation}\label{discriminant-norm-1} \Res ((N f_i)', N f_i)= N \Res (f_i', f_i)   \prod_{\substack{ 0 \leq t_1, t_2 \leq e-1 \\ t_1 \neq t_2} }  \Res( \FFrob_q^{t_1} f_i, \FFrob_q^{t_2} f_i).  \end{equation}

Plugging \eqref{resultant-norm-1} and \eqref{discriminant-norm-1} into \eqref{norm-ratio-numerator} and then applying \eqref{norm-ratio-denominator} we obtain
\begin{equation}\label{norm-poly} \begin{split} & F(N f_1,\dots, N f_r) \\ &= N F(f_1,\dots, f_r) \prod_{1 \leq i \leq r}  \Bigl(\prod_{\substack{ 0 \leq t_1, t_2 \leq e-1 \\ t_1 \neq t_2} }  \Res( \FFrob_q^{t_1} f_i, \FFrob_q^{t_2} f_i)\Bigr)^{M_{ii}}  \prod_{1 \leq i<j\leq r} \Bigl( \prod_{\substack{ 0 \leq t_1, t_2 \leq e-1 \\ t_1 \neq t_2} }  \Res( \FFrob_q^{t_1} f_i, \FFrob_q^{t_2} f_j)\Bigr)^{M_{ij} } . \end{split} \end{equation}
 Specializing to $f_i = (T-x)^{d_i}$ we have \begin{equation}\label{resultant-norm} \begin{split} & \prod_{\substack{ 0 \leq t_1, t_2 \leq e-1 \\ t_1 \neq t_2} }  \Res( \FFrob_q^{t_1} f_i, \FFrob_q^{t_2} f_j)  = \prod_{\substack{ 0 \leq t_1, t_2 \leq e-1 \\ t_1 \neq t_2} }   \Res( \FFrob_q^{t_1} (T-x)^{d_i} , \FFrob_q^{t_2} (T-x)^{d_j} )\\  = &\prod_{\substack{ 0 \leq t_1, t_2 \leq e-1 \\ t_1 \neq t_2}  } (\FFrob_q^{t_2}  x - \FFrob_q^{t_1} x)^{d_i d_j} = \Res(\pi', \pi)^{d_id_j}.\end{split}\end{equation}

By similar logic \begin{equation}\label{discriminant-norm} \prod_{\substack{ 0 \leq t_1, t_2 \leq e-1 \\ t_1 \neq t_2} }  \Res( \FFrob_q^{t_1} f_i, \FFrob_q^{t_2} f_i) = \Res(\pi', \pi)^{d_i^2}.\end{equation}

In combination, \eqref{norm-poly},  \eqref{resultant-norm}, and \eqref{discriminant-norm} demonstrate that $F(N f_1,\dots, N f_r)$ is equal to $N F(f_1,\dots, f_r) $ times a polynomial function whose value at $((T-x)^{d_1},\dots,(T-x)^{d_{r}})$ is 
\begin{equation}\label{big-norm-comparison}  \Res(\pi',\pi)^{ \sum_{i=1}^r M_{ii}    d_i^2  + \sum_{1 \leq i< j \leq r} M_{ij} d_i d_j }\neq 0. \end{equation} 
Thus, by Lemma \ref{basic-IC}(1) and \eqref{big-norm-comparison} we have
\begin{equation}\label{axiom-iii-presimplified} \begin{split} & \operatorname{tr} \left(\Frob_q,  (norm^* K_{en_1,\dots, en_{r}})_{ ((T-x) ^{d_1},\dots, (T-x) ^{d_r})} \right) \\ =&  \operatorname{tr}  \left(( \Frob_q,   ( WR_{\mathbb F_{q^e}}^{\mathbb F_q} K_{d_1,\dots,d_{r}} )_{ ((T-x)^{d_1},\dots,(T-x)^{d_{r}})}\right)\left(\frac{\pi'}{\pi}\right)_{\chi} ^{\sum_{i=1}^r M_{ii}    d_i^2  + \sum_{1 \leq i< j \leq r} M_{ij} d_i d_j } . \end{split}\end{equation}
By Lemma \ref{dilation-invariance} and Lemma \ref{scale-invariance}, unless \begin{equation}\label{simplifying-congruence} \sum_{i=1}^r M_{ii}    d_i (d_i-1)  + \sum_{1 \leq i< j \leq r} M_{ij} d_i d_j  \equiv 0 \bmod \nchi, \end{equation} the stalk of $K_{d_1,\dots,d_{r}}$ at $((T-x)^{d_1},\dots, (T-x)^{d_r})$ vanishes, which by \eqref{axiom-iii-right} means the right side of \eqref{axiom-iii-presimplified} vanishes, so the left side of \eqref{axiom-iii-presimplified} vanishes as well. It follows that 
\begin{equation}\label{axiom-iii-simplified} \begin{split} &\operatorname{tr} \left(\Frob_q,  (norm^* K_{en_1,\dots, en_{r}})_{ ((T-x) ^{d_1},\dots, (T-x) ^{d_r})} \right) \\ = &   \operatorname{tr}  \left(( \Frob_q,   ( WR_{\mathbb F_{q^e}}^{\mathbb F_q} K_{d_1,\dots,d_{r}} )_{ ((T-x)^{d_1},\dots,(T-x)^{d_{r}})}\right)\left(\frac{\pi'}{\pi}\right)_{\chi} ^{\sum_{i=1}^r M_{ii}    d_i }\end{split}. \end{equation}
since if \eqref{simplifying-congruence} is satisfied we may subtract $ \sum_{i=1}^r M_{ii}    d_i (d_i-1)  + \sum_{1 \leq i< j \leq r} M_{ij} d_i d_j $ from the exponent of $\left(\frac{\pi'}{\pi}\right)_{\chi} $ without changing the value because $\left(\frac{\pi'}{\pi}\right)_{\chi} $ is an $\nchi$'th root of unity, and if \eqref{simplifying-congruence} is not satisfied, then both sides are zero.

Combining \eqref{axiom-iii-left}, \eqref{axiom-iii-simplified}, and \eqref{axiom-iii-right}, we obtain \eqref{eq-axiom-iii}.
\end{proof}

    \begin{lemma}\label{axiom-iv}
 For each finite field $\mathbb F_q$, character $\chi$, natural numbers $d_1,\dots, d_r$, setting $d=\sum_{i=1}^r d_i$ we have
 \[ \lambda(d_1,\dots, d_r, q^m,\chi_m, M) = \sum_{j \in J(d_1,\dots,d_r; q,\chi, M)}  c_j \frac{ q^{ d}}{\overline{\alpha}_j } .\]\end{lemma}
 
 \begin{proof} If we compose the $\mathbb G_m$ action by dilation (Lemma \ref{dilation-invariance}) with the $\nchi$th power map $\mathbb G_m\to \mathbb G_m$, the factor $\lambda^{ \sum_{i=1}^r d_i (d_i-1) M_{ii} + \sum_{1 \leq i < j \leq r} d_i d_j M_{ij} } $ becomes an $\nchi$th power, and so $K_{d_1,\dots,d_{r}}$ is preserved by this $\mathbb G_m$ action by Lemma \ref{scale-invariance}. Hence the Verdier dual $DK_{d_1,\dots,d_{r}}$ is also preserved.

By Verdier duality, $H^i_c ( \prod_{i=1}^r \mathbb A^{d_i}_{\overline{\mathbb F}_q}, K_{d_1,\dots,d_{r}})$ is dual to $H^{-i} ( \prod_{i=1}^r \mathbb A^{d_i}_{\overline{\mathbb F}_q}, DK_{d_1,\dots,d_{r}})$ which by Lemma \ref{cone-contractible}, using Lemma \ref{dilation-invariance}(2) to check the condition, is $\mathcal H^{-i} ( (DK_{d_1,\dots,d_{r}})_{ ( T^{n_1},\dots,T^{n_{r}})} )$. 

Because $K_{d_1,\dots,d_r}$ is pure of weight zero on the open set where it is lisse, and $K_{d_1,\dots,d_r}[d]$ is perverse, $K_{d_1,\dots,d_r}[d]$ is perverse and pure of weight $d$, so by a theorem of Gabber \citep{FujiwaraGabber}, the trace of Frobenius on each stalk of $DK_{d_1,\dots,d_r}$ is the complex conjugate of the trace of Frobenius on the stalk of $K_{d_1,\dots,d_r}$ divided by $q^d$. Because this applies over each finite field extension, the Frobenius eigenvalues on the stalk of $DK_{d_1,\dots,d_r}$ at any point are equal to the complex conjugates, divided by $q^d$, of the Frobenius eigenvalues of $K_{d_1,\dots,d_r}$ at the same point, at least up to signed multiplicity. So the eigenvalues of $\Frob_q$ on $\mathcal H^{-i} ( (DK_{d_1,\dots,d_{r}})_{ ( T^{n_1},\dots,T^{n_{r}})} )$ are $\frac{\overline{\alpha}_j}{q^d}$, with signed multiplicities $c_j$.

By definition, the Grothendieck-Lefschetz fixed point formula, and the above isomorphisms, dualities, and eigenvalue calculations, we have
\[ \lambda(d_1,\dots, d_r; q,\chi,M) = \sum_{\substack{f_1,\dots,f_{r} \in \mathbb F_{q} [t]^+\\ \textrm{deg }f_1= d_1, \dots, \textrm{deg }f_r=d_r}}   \sum_i (-1)^i \operatorname{tr} (\Frob_{q} , \mathcal H^i (  K_{d_1,\dots,d_{r}})_{f_1,\dots, f_r })  \]
\[ = \sum_i(-1)^i \operatorname{tr} \Bigl(\Frob_{q}, H^i_c \Bigl( \prod_{i=1}^r \mathbb A^{d_i}_{\overline{\mathbb F}_q}, K_{d_1,\dots,d_{r}}\Bigr)\Bigr)  = \sum_i (-1)^i \operatorname{tr} \Bigl(\Frob_{q}^{-1} , H^{-i} \Bigl( \prod_{i=1}^r \mathbb A^{d_i}_{\overline{\mathbb F}_q}, D K_{d_1,\dots,d_{r}}\Bigr)\Bigr)  \] \[= \sum_i (-1)^i \operatorname{tr} (\Frob_{q}^{-1} , \mathcal H^{-i} ( (DK_{d_1,\dots,d_{r}})_{ ( T^{n_1},\dots,T^{n_{r}})} )))= \sum_j c_j  \left( \frac{\overline{\alpha}_j}{q^d}\right)^{-1}  = \sum_j c_j \frac{q^d}{\overline{\alpha}_j}  . \qedhere \] \end{proof}

 \begin{lemma}\label{axiom-v} 
 For each finite field $\mathbb F_q$, character $\chi$, natural numbers $d_1,\dots, d_r$, and $(\alpha_j, c_j) \in J(d_1,\dots, d_r;q,\chi, M)$ we have  $|\alpha_j| < q^{ \frac{  d -1 }{2}}$ as long as $d \geq 2$ where $d= \sum_{i=1}^r d_i$. \end{lemma} 
 
 \begin{proof}   Because $K_{d_1,\dots,d_{r}}$ is the IC sheaf of a lisse sheaf, its $\mathcal H^i$ is supported in codimension at least $i+1$ for all $i>0$ \cite[Proposition 2.1.11]{bbd}. By Lemma \ref{translation-invariance}, any stalk cohomology at a point must also occur at its one-dimensional orbit under the action of $\mathbb G_a$ by translation, hence with codimension $\leq d-1$, thus in degree $\leq d-2$, as long as $d \geq 2$. So because intersection cohomology complexes are pure, any Frobenius eigenvalues that appear are $\leq q^{\frac{d-2}{2}}$.

\end{proof}


 \begin{proof}[Proof of Theorem \ref{axiomatics}] In view of Lemmas \ref{axiom-i}, \ref{axiom-ii}, \ref{compatibility}, \ref{axiom-iii}, \ref{axiom-iv}, and \ref{axiom-v} it suffices to prove that the function $a$ is uniquely determined by these axioms.
 
 In fact we will show that $ J(d_1,\dots,d_r; q,\chi, M)$ is determined by these axioms whenever $d_1+ \dots+ d_r \leq d$, for all $d$. This will then determine $a$ by axioms (1) and (3).
 
 We do this by induction on $d$. The cases $d=0$ and $d=1$ are determined by axiom (2) and the fact that there is at most one way of expressing a given function of a natural number $m$ as a finite signed sum of $m$th powers.
 
 For the induction step, assume that $ J(d_1,\dots,d_r; q,\chi, M)$ is determined by these axioms whenever $d_1+ \dots+ d_r < d$. From axiom (3), this determines $a(\pi^{d_1},\dots, \pi^{d_r};q,\chi,M)$ whenever $d_1+\dots + d_r < d$. From axiom (1), this determines $a(f_1,\dots,f_r;q,\chi,M)$ whenever each prime factor of $\prod_{i=1}^r f_i$ occurs with multiplicity less than $d$.
 
 Thus, if $\deg f_i = d_i$ and $d_1+\dots + d_r= d$, the axioms determine $a(f_1,\dots,f_r;q,\chi,M)$ when $\prod_{i=1}^r f_i$ is not a $d$th power of a linear prime, i.e. in all cases but when $f_i$ is of the form $(T-x)^{d_i} $ for all $i$. By axioms (3) and (4) applied to $\mathbb F_{q^m}$ and $\chi_m$ we have 
 \[ \sum_{\substack{f_1,\dots,f_{r} \in \mathbb F_{q^m}[t]^+\\ \textrm{deg }f_i= d_i}}  a(f_1,\dots,f_{r}; q^m, \chi_m,M) - \sum_{x \in \mathbb F_{q^m} } a((T-x) ^{d_1},\dots, (T-x) ^{d_r}; q^m, \chi_m, M) \] \[= \sum_{j \in J(d_1,\dots,d_r; q^m,\chi_m, M)}  c_j  \frac{ q^{ \sum_{i=1}^{r} d_i}}{\overline{\alpha}_j }  -q^m \sum_{j \in J(d_1,\dots,d_r; q^m,\chi_m, M)}  c_j \alpha_j .\]
 
 However, by the compatibility of $J$, $J(d_1,\dots,d_r; q^m,\chi_m, M)$ consists of the $m$th powers of $J(d_1,\dots,d_r; q,\chi, M)$ so we obtain 
\begin{equation}\label{induction-formula} \begin{split}& \sum_{\substack{f_1,\dots,f_{r} \in \mathbb F_q^m[t]^+\\ \textrm{deg }f_i= d_i}}  a(f_1,\dots,f_{r}; q^m, \chi_m,M) - \sum_{x \in \mathbb F_{q^m}} a((T-x) ^{d_1},\dots, (T-x)^{d_r}; q^m, \chi_m, M)  \\ = & \sum_{j \in J(d_1,\dots,d_r; q,\chi, M)}  c_j \left( \frac{ q^{ \sum_{i=1}^{r} d_i}}{\overline{\alpha}_j }\right)^{m }   - \sum_{j \in J(d_1,\dots,d_r; q,\chi, M)}  c_j (q\alpha_j )^m .\end{split}\end{equation}
 
 We have already shown the left side of \eqref{induction-formula} is determined by the axioms for all $m$. The right side of \eqref{induction-formula} is a finite sum of $m$th powers of Weil numbers, so the Weil numbers appearing, and their multiplicity, are uniquely determined by the left side of \eqref{induction-formula}. The only difficulty is whether any given Weil number occurs in the first term or the second term. However, by axiom (5), $q \alpha_j$ appears in the second term only if $|\alpha_j | < q^{ (d-1)/2}$, so $|q \alpha_j| <q^{(d+1)/2}$, while $q^d/\alpha_j$ appearing in the first term satisfies $|q^d/\alpha_j|> q^{ (d+1)/2}$, so each Weil number can only appear in one of the two terms, thus both terms are uniquely determined. \end{proof}

 \begin{cor}\label{axiomatics-2} Fix $M$, $w_1,\dots,w_r \in \mathbb Z$, $\epsilon_1,\dots,\epsilon_r \in \{0,1\}$. Fix for each $i$ with $\epsilon_i=0$ a compatible system of Weil numbers $\gamma_i$ and for each $i$ with $\epsilon_i=1$ a sign-compatible system of Weil numbers $\gamma_i$. In either case, assume that $|\gamma_i(q,\chi)| = q^{ w_i/2}$. Let  
   \[ a^* (f_1,\dots, f_r; q ,\chi,M)= a (f_1,\dots, f_r; q ,\chi,M) \prod_{i=1}^r \gamma_i(q,\chi)^{\deg f_i}.\] 
Then  
  \[ a^* (f_1,\dots, f_r; q ,\chi,M) \] is the unique function that, together with a function $J^*(d_1,\dots, d_r; q,\chi, M)$ from tuples of natural numbers $d_1,\dots, d_r$, to compatible systems of sets of ordered pairs of Weil numbers and integers, satisfies the axioms 

\begin{enumerate}

\item If $f_1,\dots,f_{r}$ and $g_1,\dots, g_{r}$ satisfy $\gcd(f_i,g_j)=1$ for all $i$ and $j$, then we have \[ a^* ( f_1g_1,\dots, f_{r} g_{r}; q , \chi, M) \] \[= a^* (f_1,\dots,f_{r} ; q, \chi, M) a^*(g_1,\dots, g_{r} ; q, \chi,  M)  \prod_{1 \leq i \leq r} \left( \frac{f_i}{g_i} \right)_{\chi}^{M_{ii}} \left( \frac{g_i}{f_i} \right)_{\chi}^{M_{ii}}   \prod_{1 \leq i < j \leq r}  \left( \frac{f_i}{g_j} \right)_{\chi}^{M_{ij}}\left( \frac{g_i}{f_j} \right)_{\chi}^{M_{ij}} .\] 

\item $a^*(1,\dots,1;q,\chi,M)=1$ and $a^*(1,\dots, 1, f, 1,\dots,1;q,\chi,M)=\gamma_i(q,\chi) $ for all linear polynomials $f$. 

\item 
   \[ a^*(\pi^{d_1},\dots, \pi^{d_r}; q, \chi, M) =\left(\frac{\pi'}{\pi}\right)_{\chi} ^{ \sum_{i=1}^r { d_i M_{ii}} }  (-1)^{ \sum_{i=1}^r \epsilon_i d_i (\deg \pi +1) }  \sum_{j \in J(d_1,\dots,d_r; q,\chi, M)}  c_j \alpha_j^{\deg \pi}  .\]
    \item
 \[ \sum_{\substack{f_1,\dots,f_{r} \in \mathbb F_q[t]^+\\ \textrm{deg }f_i= d_i}}  a^*(f_1,\dots,f_{r}; q^m, \chi,M)= \sum_{j \in J(d_1,\dots,d_r; q,\chi, M)}  c_j \frac{ q^{ \sum_{i=1}^{r} (1+w_i) d_i}}{\overline{\alpha}_j }.\]
 
 \item  $|\alpha_j| < q^{ \frac{  \sum_{i=1}^{r} (1+w_i) d_i -1 }{2}}$ as long as $\sum_{i=1}^{r} d_i \geq 2$. 
 
 \end{enumerate}
  
 \end{cor} 
  
  \begin{proof} This follows from Theorem \ref{axiomatics} once we check that $a^*( f_1,\dots, f_r;q, \chi, M)$ satisfies these axioms with a given $J(d_1,\dots, d_r; q, \chi, M)$  if and only if
  \[ \tilde{a}(f_1,\dots, f_r; q,\chi, M)=  \frac{ a^*( f_1,\dots, f_r;q, \chi, M)}{ \prod_{i=1}^r \gamma_i(q,\chi)^{\deg f_i}}\]
satisfies the axioms of Theorem \ref{axiomatics} after adjusting $J^*(d_1,\dots,d_r; q,\chi,M)$ by dividing each $\alpha_j$ by $\prod_{i=1}^{r} ((-1)^{\epsilon_i} \gamma_i(q,\chi))^{d_i}$ and each $ c_j$ by $(-1)^{ \sum_{i=1}^r \epsilon_i d_i}$.

This can be checked one axiom at a time by plugging these expressions into each axiom of Theorem \ref{axiomatics}, simplifying, and observing that they match the corresponding axiom here, as well as deducing the compatibility of $J$ from the (sign-)compatibility of $\gamma_i$.

In each case, this is relatively straightforward.  In (4) it requires the identity $\gamma_i(q,\chi) \overline{\gamma_i(q,\chi)}=q^{w_i}$.


  
  \end{proof}

\section{Examples}

For some special values of $M$, we can calculate $a$ by exhibiting an explicit function and checking that it satisfies the axioms of Theorem \ref{axiomatics}. In fact, these will be functions $a$ that have essentially appeared in the literature already as coefficients of multiple Dirichlet series, and most of the properties described in Theorem \ref{axiomatics} were previously observed (but in slightly different language, so we will have to do some work to match it up). In some cases, it will also be convenient to use additional geometric techniques to calculate $a$.

One reason for the difference in language is that prior work has tended to define twisted multiplicative functions as the product of a multiplicative function with a Dirichlet character. We have found it more convenient to define twisted multiplicative functions all at once.

We will always use $\tilde{a}$ to refer to a function we are trying to prove satisfies the axioms of Theorem \ref{axiomatics}, but haven't yet.

\begin{prop}\label{easiest-example}  Take $r=2$, $M= \begin{pmatrix} 0 & 1 \\ 1 & 0 \end{pmatrix} $. 

Then \[a(f_1,f_2; q,\chi, M) =  \begin{cases} \left( \frac{f_1/g^{\nchi}} {f_2/g^{\nchi}} \right)_\chi q^{ (\nchi-1) \deg g} & \textrm { if }\gcd(f_1,f_2) = g^{\nchi}\textrm{ for some }$g$ \\
0 & \textrm{ if }\gcd(f_1,f_2)\textrm { is not an }\nchi\textrm{th power} \end{cases}.\] \end{prop}

We prove this after making some definitions. Let 
 \[\tilde{a}\left(f_1,f_2; q,\chi,  \begin{pmatrix} 0 & 1 \\ 1 & 0 \end{pmatrix}\right) =  \begin{cases} \left( \frac{f_1/g^{\nchi}} {f_2/g^{\nchi}} \right)_\chi q^{ (\nchi-1) \deg g} & \textrm { if }\gcd(f_1,f_2) = g^{\nchi}\textrm{ for some }$g$ \\
0 & \textrm{ if }\gcd(f_1,f_2)\textrm { is not an }\nchi\textrm{th power} \end{cases}.\] 

In \citep[(1.2)]{ChintaMohler}, a function $a$ is defined to be the unique multiplicative function such that \[ a (\pi^j, \pi^k) = \begin{cases} p^{ (n-1) \min(j,k)/n} & \textrm{if } \min(j,k) =0 \bmod n \\ 0 & \textrm{otherwise} \end{cases}.\] Furthermore they define $ f_{2,0} $ as quotient of $f_2$ by its maximal $\nchi$th power divisor and $\hat{f}_1$ as the greatest divisor of $f_1$ coprime to $f_{2,0}$. They define a Dirichlet series with coefficients \[ \left( \frac{ \hat{f}_1}{ f_{2,0}}\right)_\chi a(f_1,f_2).\]

\begin{lemma}\label{easiest-example-cm-1} For all finite fields $\mathbb F_q$, characters $\chi$, and monic polynomials $f_1,f_2$ over $\mathbb F_q$, we have
\[  \tilde{a}\left(f_1,f_2; q,\chi,  \begin{pmatrix} 0 & 1 \\ 1 & 0 \end{pmatrix}\right)=  \left( \frac{ \hat{f}_1}{ f_{2,0}}\right)_\chi a(f_1,f_2).\]
\end{lemma}

\begin{proof} First we note that ${a}(f_1,f_2)$ vanishes unless $\gcd(f_1,f_2 )=g^n$ for some $g$ and is $q^{ (n-1) \deg g}$ in that case. So it suffices to check, when $\gcd(f_1,f_2) = g^n$, that
\[ \left( \frac{f_1/g^n} {f_2/g^n} \right)_\chi =  \left( \frac{ \hat{f}_1}{ f_{2,0}}\right)_\chi .\]

First note that $f_{2,0}$ divides $f_2/g^n$ and the ratio is an $\nchi$th power which is prime to $f_1/g^n$, so we have \[  \left( \frac{f_1/g^n} {f_2/g^n} \right)_\chi = \left( \frac{f_1/g^n}{ f_{2,0}} \right)_{\chi} .\]

Now $\hat{f}_1$ is the quotient of $f_1$ by a product of $\pi^{ v_\pi(f_1)}$, where $\pi$ are some primes. Each such $\pi$ divides $f_{2,0}$, so $v_\pi(f_2)$ cannot be  multiple of $\nchi$.  Since  $v_\pi(\gcd(f_1,f_2)) =\min (v_\pi(f_1),v_\pi(f_2))$ is a multiple of $\nchi$, we must have $v_\pi(f_1)$ a multiple of $\nchi$ strictly less than $v_\pi(f_2)$. Thus $f_1/ \hat{f}_1$ is an $\nchi$th power and divides $g^\nchi$, so $\hat{f}_1$ is a multiple of $f_1/g^\nchi$ by an $\nchi$th power prime to $f_{2,0}$. Thus
\[  \left( \frac{f_1/g^n}{ f_{2,0}} \right)_{\chi} =  \left( \frac{ \hat{f}_1}{ f_{2,0}}\right)_\chi \] and we are done. \end{proof}

\begin{proof}[Proof of Proposition \ref{easiest-example}] It suffices to prove that $\tilde{a}$ satisfies the axioms of Theorem \ref{axiomatics}. Axiom (2) is immediate. To check $\tilde{a}$  satisfies axiom (1), observe that if $\gcd(f_i,g_j)=1$ for all $i,j$ then $\gcd(f_1g_1,f_2g_2) = \gcd(f_1,f_2) \gcd(g_1,g_2)$, and moreover the two gcds on the right are coprime, so  $\gcd(f_1g_1,f_2g_2) $ is an $\nchi$th power if and only if both $ \gcd(f_1,f_2)$ and $ \gcd(g_1,g_2)$ are.

We next choose $J(d_1,d_2;q , \chi , M)$. We observe that $\tilde{a} ( \pi^{d_1}, \pi^{d_2}; q,\chi,M)$ vanishes unless $\min(d_1,d_2)$ is divisible by $\nchi$ and equals $q^{ (n-1) \deg \pi \min(d_1,d_2)/n}$ in that case. Hence we can take $J(d_1,d_2; q,\chi,M)$ to be empty unless $\min(d_1,d_2)$ is divisible by $\nchi$ and to consist of the ordered pair $(q^{ (n-1) \min(d_1,d_2)/n},1)$ if it is divisible.

This makes (3) immediate. (5) is similarly clear. With this value of $J$, (4) is equivalent to the statement that
\[ \sum_{ \substack{ f_1 ,f_2 \in \mathbb F_q[t]^+ \\ \deg(f_1)=d_1 \\ \deg (f_2)=d_2}} \tilde{a}(f_1,f_2; q , \chi, M ) =  \begin{cases} q^{ d_1 + d_2 - \frac{n-1}{n} \min(d_1,d_2)} & \textrm{ if } n| \min(d_1,d_2)  \\ 0 & \textrm{otherwise}  \end{cases} .\] 
We now use \citep[(1.7)]{ChintaMohler}, which is
\[ \sum_{ f_1,f_2 \in \mathbb F_q[t]^+}   \left( \frac{ \hat{f}_1}{ f_{2,0}}\right)_\chi a(f_1,f_2) x^{\deg f_1} y^{\deg f_2} =  \frac{ 1-q^2 xy}{ (1-qx) (1-qy) (1 - q^{\nchi+1} x^{\nchi}y^{\nchi}) }.\]

Thus, by Lemma \ref{easiest-example-cm-1},
\begin{equation}\label{easiest-example-gf} \sum_{ f_1,f_2 \in \mathbb F_q[t]^+}   \tilde{a}(f_1,f_2; q,\chi, M) x^{\deg f_1} y^{\deg f_2} =  \frac{ 1-q^2 xy}{ (1-qx) (1-qy) (1 - q^{\nchi+1} x^{\nchi} y^{\nchi}) }.\end{equation}

Then \[ \sum_{ \substack{ f_1 ,f_2 \in \mathbb F_q[t]^+ \\ \deg(f_1)=d_1 \\ \deg (f_2)=d_2}} \tilde{a}(f_1,f_2)  \] is simply the coefficient of $x^{d_1} y^{d_2}$ in \eqref{easiest-example-gf}. Hence to verify (4) it suffices to check that 
\[ \frac{ 1-q^2 xy}{ (1-qx) (1-qy) (1 - q^{n+1} x^{\nchi} y^{\nchi}) } = \sum_{\substack{ d_1,d_2 \in \mathbb N \\ \min(d_1,d_2) \equiv 0\bmod n}}   q^{ d_1 + d_2 - (\nchi-1) \min(d_1,d_2)/\nchi}  x^{d_1} y^{d_2} \]
which is straightforward.\end{proof}

\begin{cor}\label{easiest-example-cm-2} For all finite fields $\mathbb F_q$, characters $\chi$, and monic polynomials $f_1,f_2$ over $\mathbb F_q$, we have
\[  {a}(f_1,f_2; q,\chi,  \begin{pmatrix} 0 & 1 \\ 1 & 0 \end{pmatrix})=  \left( \frac{ \hat{f}_1}{ f_{2,0}}\right)_\chi a(f_1,f_2).\]
\end{cor}
\begin{proof} This follows from combining Lemma \ref{easiest-example-cm-1} and Proposition \ref{easiest-example}. \end{proof}

\begin{prop}\label{gauss-example} Assume $\nchi$ even. 

 Take $r=2$, $M= \begin{pmatrix} 0 & -1 \\ -1 & \frac{n}{2}+1  \end{pmatrix} $. 

Then \begin{equation}\label{eq-gauss-example} a (f_1,f_2; q ,\chi , M ) =   \frac{(-1)^{ \frac{ \deg f_2 (\deg f_2-1) (q-1)}{4}}  }{  G(\chi,\psi)^{\deg f_2}}   \sum_{ \substack{ u \in \mathbb F_q[t]^+ \\ u^n | f_2}}  q^{ (n-1)  \deg u} g_\chi(f_1, f_2/u^n ) . \end{equation}

\end{prop}

Let \[\tilde{a }\left(f_1,f_2; q ,\chi , \begin{pmatrix} 0 & -1 \\ -1 & \frac{n}{2}+1  \end{pmatrix}\right ) =   \frac{(-1)^{ \frac{ \deg f_2 (\deg f_2-1) (q-1)}{4}}  }{  G(\chi,\psi)^{\deg f_2}}   \sum_{ \substack{ u \in \mathbb F_q[t]^+ \\ u^n | f_2}}  q^{ (n-1)  \deg u} g_\chi(f_1, f_2/u^n ) .\] 

We give two proofs of Proposition \ref{gauss-example}. The first uses geometric properties of perverse sheaves, while the second relies on Theorem \ref{axiomatics} and \citep{ChintaMohler}.

The geometric proof proceeds by a series of lemmas that establish \eqref{eq-gauss-example} in successively more cases.

\begin{lemma}\label{ge-geom-1}  \eqref{eq-gauss-example} holds when $f_2$ is squarefree and $f_1$ and $f_2$ are coprime. \end{lemma}
\begin{proof} The trace function of $\mathcal L_\chi(f)$ is $\chi(f)$ on any open set where $f$ is a nonvanishing polynomial function. On the open set where $f_2$ is squarefree and $f_1$ and $f_2$ are coprime, $F_{d_1,d_2}$ is a nonvanishing polynomial (the condition that $f_1$ is squarefree being unneeded since $M_{11}=0$) so   \[a(f_1,f_2;q,\chi,M)=\chi( F_{d_1,d_2} (f_1,f_2) )= \left( \frac{f_2'}{f_2} \right)_\chi^{n/2+1}   \left( \frac{f_2}{f_1} \right)_\chi ^{-1}\]  by Lemma \ref{resultant}.

On the other hand, when $f_2$ is squarefree we have \[  \sum_{ \substack{ u \in \mathbb F_q[t]^+ \\ u^n | f_2}}  q^{ (n-1)  \deg u} g_\chi(f_1, f_2/u^n ) = g_\chi(f_1,f_2)\] and by Lemma \ref{gauss-evaluation},
\[  \frac{(-1)^{ \frac{ \deg f_2 (\deg f_2-1) (q-1)}{4}}  }{  G(\chi,\psi)^{\deg f_2}} g_\chi(f_1,f_2) =  \left(\frac{ f_2'}{f_2} \right)_{\chi} \left(\frac{f_2'}{f_2} \right)_{\xi}\left(\frac{f_1}{f_2} \right)_{\chi}^{-1} \]
so \eqref{eq-gauss-example} follows upon noting that \[ \left(\frac{f_2'}{f_2} \right)_\chi \left(\frac{ f_2'}{f_2} \right)_{\xi} = \left( \frac{f_2'}{f_2} \right)_\chi^{n/2+1}   .\]
\end{proof}

\begin{lemma}\label{ge-geom-2} \eqref{eq-gauss-example} holds when $\deg f_1 \geq \deg f_2$. \end{lemma}

\begin{proof}
Let $M' = \begin{pmatrix} 0 & 1 \\ 1 & 0 \end{pmatrix}$. Observe that \[\sum_{ \substack{ u \in \mathbb F_q[t]^+ \\ u^n | f_2}}  q^{ (n-1)  \deg u} g_\chi(f_1, f_2/ u^n)  = \sum_{ \substack{ u \in \mathbb F_q[t]^+ \\ u^n | f_2}}  q^{ (n-1)  \deg u} \sum_{\substack{ h \in \mathbb F_q[t]/ f_2 \\ u^n | h} } \left( \frac{h/u^n}{f_2/u^n} \right)_\chi \psi \left( \operatorname{res} \left(\frac{h f_1}{ f_2} \right) \right)\]
\[  =  \sum_{  \substack{ h \in \mathbb F_q[t]^+ \\ \deg h = \deg f_2}}  \sum_{ \substack{ u \in \mathbb F_q[t]^+ \\ u^n \mid h, f_2}}  q^{ (n-1)  \deg u} \left( \frac{h/u^n}{f_2/u^n} \right)_\chi \psi \left( \operatorname{res} \left(\frac{h f_1}{ f_2} \right) \right)= \sum_{  \substack{ h \in \mathbb F_q[t]^+ \\ \deg h = \deg f_2}}  a\left( h,f_2; q ,\chi , M' \right) \psi \left( \operatorname{res} \left(\frac{h f_1}{ f_2} \right) \right)\] using the fact that there is a unique monic $h$ of degree $\deg f_2$ in each residue class mod $f_2$, the fact that $\left( \frac{h/u^n}{f_2/u^n} \right)_\chi =0$ unless $u^n=\gcd(h,f_2)$, and Lemma \ref{easiest-example}. Thus
\begin{equation}\label{fourier-for-fe} \tilde{a }\left(f_1,f_2; q ,\chi , \begin{pmatrix} 0 & -1 \\ -1 & \frac{n}{2}+1  \end{pmatrix}\right ) =   \frac{(-1)^{ \frac{ \deg f_2 (\deg f_2-1) (q-1)}{4}}  }{  G(\chi,\psi)^{\deg f_2}}   \sum_{  \substack{ h \in \mathbb F_q[t]^+ \\ \deg h = \deg f_2}}  a\left( h,f_2; q ,\chi,  M' \right) \psi \left( \operatorname{res} \left(\frac{h f_1}{ f_2} \right) \right).\end{equation}

Let $d_1=\deg f_1$ and $d_2 = \deg f_2$.

We now make a geometric argument.  To distinguish IC sheaves constructed with the matrix $M'$ from those constructed with the matrix $M$, we put the matrix as an additional subscript. Thus $K_{d_2,d_2,  M'}$ is a complex of sheaves on $\mathbb A^{d_2} \times \mathbb A^{d_2}$ whose trace function is $a( h,f_2; q ,\chi  ,M' ) $ and $K_{d_1,d_2,M}$ is a complex of sheaves on $\mathbb A^{d_1} \times \mathbb A^{d_2}$ whose trace function is $a( h,f_2; q ,\chi  ,M)$.

We now recall the $\ell$-adic Fourier transform defined by \citet{KL}. We define two maps $\mathbb A^{d_2} \times \mathbb A^{d_2} \times \mathbb A^{d_2} \to \mathbb A^{d_2} \times \mathbb A^{d_2}$, namely $pr_{13}$ and $pr_{23}$, given respectively by projection onto the first and third factors, and projection onto the second and third factors. We also define a map $\mu \colon \mathbb A^{d_2} \times \mathbb A^{d_2} \times \mathbb A^{d_2}\to \mathbb A^1$ given by taking the dot product of the first and second factors.  Precisely in coordinates, we think of points of the second $\mathbb A^{d_2}$ as parameterizing monic polynomials $t^n+ \sum_{i=0}^{d_2-1} c_i t^i  $, points of the first $\mathbb A^{d_2}$ as simply tuples $b_0, \dots, b_{d_2-1}$, and $\mu( (b_0,\dots, b_{d_2-1}), (c_0,\dots, c_{d_2-1}), f_2 ) = \sum_{i=0}^{d_2-1} b_i c_i$.

\citet[(2.1.1)]{KL} define the Fourier transform $\mathcal F_\psi $ by the formula  
\[ \mathcal F_{\psi} K_{d_2,d_2,  M'}= pr_{13!} ( pr_{23}^*K_{d_2,d_2,  M'} \otimes \mu^* \mathcal L_\psi ) [d_2].\]
The operations of pullback, compactly supported pushforward, tensor product, and shift each transform the trace function in a predictable way. Using this, it is immediate that the trace function of $\mathcal F_{\psi} K_{d_2,d_2,  M'}$ at a point $(\mathbf{b}, f_2)$ of $\mathbb A^{d_2} \times \mathbb A^{d_2}$ is given by the formula
\[(-1)^{d_2}  \sum_{  \substack{ h \in \mathbb F_q[t]^+ \\ \deg h = \deg f_2}}  a\left( h,f_2; q ,\chi , M' \right) \psi \left( h \cdot \mathbf b \right).\] 

Let $\sigma\colon \mathbb A^{d_1} \times \mathbb A^{d_2} \to \mathbb A^{d_1} \times \mathbb A^{d_2}$ be the map sending $(f_1,f_2)$ to $(\mathbf b, f_2)$ where $b_i =  \operatorname{res} \left( \frac{ t^i f_1}{ f_2} \right)$. Let $\alpha \colon  \mathbb A^{d_1} \times \mathbb A^{d_2} \to \mathbb A^1 $ send $(f_1,f_2)$ to $\operatorname{res} \left( \frac{t^{d_2 } f_1}{ f_2} \right)$. We have chosen these so that for $(\mathbf b, f_2)=\sigma(f_1,f_2)$, we have 
\[\alpha(f_1,f_2) + h \cdot \mathbf b = \operatorname{res} \left( \frac{t^{d_2} f_1}{ f_2} \right) + \sum_{i=0}^{d_2-1} c_i  \operatorname{res} \left( \frac{t^i f_1}{ f_2} \right)=  \operatorname{res} \left( \frac{( t^{d_2} + \sum_{i=0}^{n-1} c_i t^i )  f_1}{ f_2} \right) = \operatorname{res} \left(\frac{h f_1}{f_2} \right).\]

Thus the trace function of 
\[ \sigma^*\mathcal F_{\psi} K_{d_2,d_2,  M'} \otimes \alpha^* \mathcal L_\psi \] is given by
\[ (-1)^{d_2} \sum_{  \substack{ h \in \mathbb F_q[t]^+ \\ \deg h = \deg f_2}}  a\left( h,f_2; q ,\chi , M' \right) \psi \left( \operatorname{res} \left(\frac{h f_1}{ f_2} \right) \right) = \frac{  (-G(\chi,\psi))^{d_2}}
 {(-1)^{ \frac{ d_2 (d_2-1) (q-1)}{4}}  } \tilde{a }\left(f_1,f_2; q ,\chi , M\right ).\]
 
 By the Hasse-Davenport relations, the quantity $-G(\chi,\psi)$ is a compatible system of Weil numbers, and the same is true for $(-1)^{\frac{q-1}{2}}= \xi(-1)$, so there exists a sheaf $\mathcal L_G$ on $\operatorname{Spec} \mathbb F_p$ whose trace of $\operatorname{Frob}_q$ is $ \frac
 {(-1)^{ \frac{ d_2 (d_2-1) (q-1)}{4}}  }{  (-G(\chi,\psi))^{d_2}}$ for all finite fields $\mathbb F_q$. It follows that the trace function of $\sigma^*\mathcal F_{\psi} K_{d_2,d_2,  M'} \otimes \alpha^* \mathcal L_\psi \otimes \mathcal L_G$ is $ \tilde{a }(f_1,f_2; q ,\chi , M )$.
 
Next let's check that $ \sigma^*\mathcal F_{\psi} K_{d_2,d_2,  M'} \otimes \alpha^* \mathcal L_\psi \otimes \mathcal L_G [d_1+ d_2]$ is an irreducible perverse sheaf. The complex $K_{d_2,d_2,  M'} [2 d_2]$ is perverse by construction. Fourier transform preserves perversity by the same argument as \cite[Corollary 2.1.5(iii)]{KL}, which shows Fourier transform preserves relative perversity, and preserves irreducibility by an immediate consequence of \cite[III, Theorem 8.1(3)]{Kiehl-Weissauer}. We can check that $\sigma$ is smooth because for each fixed value of $f_2$,  $\sigma$ is given by a linear map of vector spaces, and this linear map is surjective because $(h,f_1) \mapsto \operatorname{res} \left(\frac{h f_1}{f_2}\right)$ gives a perfect pairing on polynomials modulo $f_2$. This also shows that $\sigma$ has nonempty, geometrically connected fibers. Since the source of $\sigma$ has dimension $d_1+d_2$ and the target has dimension $2d_2$, the map $\sigma$ must be smooth of relative dimension $d_1-d_2$, so $\sigma$ preserves perversity after a shift by $d_1-d_2$. Because $\sigma$ has nonempty, geometrically connected fibers, this pullback and shift functor is fully faithful and thus preserves irreducibility  \citep*[Corollary 4.2.6.2]{bbd}.  Finally, $\mathcal L_\psi$ is lisse of rank one so its pullback under $\alpha$ is lisse of rank one and tensor product with it preserves perversity and irreducibility, and the same is true for $\mathcal L_G$.

So $ \sigma^*\mathcal F_{\psi} K_{d_2,d_2, M'} \otimes \alpha^* \mathcal L_\psi \otimes \mathcal L_G[d_1+ d_2]$ and $K_{d_1,d_2,M}[d_1+d_2]$ are two irreducible perverse sheaves. Furthermore, by Lemma \ref{ge-geom-1}, the trace functions of these two perverse sheaves agree on the open set where $f_1$ is squarefree and $f_1$ and $f_2$ are coprime by Lemma \ref{ge-geom-1}. Restricting to a possibly smaller open set where both are lisse, we get two irreducible lisse sheaves with the same trace function, which must be isomorphic. Since $K_{d_1,d_2,M}$ is lisse of nonzero rank on an open set, both sheaves are lisse of nonzero rank, and because they are irreducible, must be middle extensions from that open set. Since both are the middle extension of the same lisse sheaf from the same open set, they are isomorphic as perverse sheaves. It follows that these two irreducible perverse sheaves have the same trace function, giving \eqref{eq-gauss-example}.  \end{proof}

%

\begin{proof}[Conclusion of geometric proof of Proposition \ref{gauss-example}] 
Given $f_1,f_2$, find $v$ coprime to $f_2$ and such that $\deg f_1 + \deg v \geq \deg f_2$, and compute using axiom (1) and the fact that $K_{\deg v,0}$ is the constant sheaf that
\begin{equation}\label{gauss-example-geometric-general-1} a(f_1v,f_2; q,\chi,M) = a(f_1,f_2;q,\chi,M) a(v,1;q,\chi,M)  \left( \frac{v}{f_2}\right)_\chi^{-1} = a(f_1,f_2;q,\chi,M)   \left( \frac{v}{f_2}\right)_\chi^{-1} .\end{equation}
By Lemma \ref{ge-geom-2} we have  \begin{equation}\label{gauss-example-geometric-general-2} a (f_1v,f_2; q ,\chi , M ) =   \frac{(-1)^{ \frac{ \deg f_2(\deg f_2-1) (q-1)}{4}}  }{  G(\chi,\psi)^{\deg f_2}}   \sum_{ \substack{ u \in \mathbb F_q[t]^+ \\ u^n | f_2}}  q^{ (n-1)  \deg v} g_\chi(f_1v, f_2/u^n) .\end{equation}

But by Lemma \ref{gauss-multiplicative-easy},  \begin{equation}\label{gauss-example-geometric-general-3} g_\chi(f_1v, f_2/u^n) = g_\chi(f_1, f_2/u^n) \left( \frac{v}{f_2/u^n}\right)_{\chi}^{-1} = g_\chi(f_1, f_2/u^n)  \left( \frac{v}{f_2}\right)_{\chi}^{-1} \end{equation}
so combining \eqref{gauss-example-geometric-general-1},  \eqref{gauss-example-geometric-general-2},  and \eqref{gauss-example-geometric-general-3}, we get
\[ a(f_1,f_2;q,\chi,M)   \left( \frac{v}{f_2}\right)_\chi^{-1} =   \frac{(-1)^{ \frac{ \deg f_2(\deg f_2-1) (q-1)} {4}}  }{  G(\chi,\psi)^{\deg f_2}}   \sum_{ \substack{ u \in \mathbb F_q[t]^+ \\ u^n | f_2}}  q^{ (n-1)  \deg u} \left( \frac{v}{f_2}\right)_{\chi}^{-1}  g_\chi(f_1, f_2/u^n) .\] 
and dividing both sides by $\left( \frac{v}{f_2}\right)_{\chi}^{-1} $ we get \eqref{eq-gauss-example} in general. \end{proof}
 
 Before performing our proof using \citep{ChintaMohler}, we will explain the relationship of the Gauss sums we work with to the formula defined by \cite{ChintaMohler}.
 
 To do this, we use notation from \citep{ChintaMohler}. They define a function $b$ as the unique multiplicative function satisfying

\[ b(\pi^{d_1}, \pi^{d_2} ) =  \begin{cases} 1 & \textrm{if } d_2=0 \\
(q^{\deg \pi } -1)  q^{ (d_2/2 -1 ) \deg \pi} & \textrm{if }d_2 \equiv 0 \bmod n\textrm{ and }d_1 \geq d_2 \\
0  & \textrm{if }d_2 \not\equiv 0 \bmod n\textrm{ and }d_1 \geq d_2>0 \\
- q^{(d_2/2-1 ) \deg \pi}  & \textrm{if } d_1 = d_2 -1\textrm{ and } d_2 \equiv 0 \bmod n \\
q^{(d_2-1 ) \deg \pi/2}  & \textrm{if } d_1 = d_2 -1\textrm{ and } d_2 \not\equiv 0 \bmod n \\
0 & \textrm{if } d_1< d_2 -1 \end{cases} \]

Let $f_{2,0}$ be $f_2$ divided by the greatest $\nchi$th power that divides $f_2$. Let $f_{2,\flat}$ be the largest squarefree divisor of $f_{2,0}$ and let $\hat{f}_1$ be the largest divisor of $f_1$ prime to $f_{2,0}$. Equivalently,  \[ f_{2,0} = \prod_{\substack{ \pi \textrm{ prime } \\  \pi \mid f_2 } } \pi^{ v_\pi(f_2) -  \nchi \lfloor \frac{v_\pi(f_2) }{\nchi}\rfloor }  \hspace{10pt} \textrm{ and } \hspace{10pt}   \hat{f}_1 = \prod_{\substack{ \pi \textrm{ prime } \\  \pi \mid f_1 \\ \nchi \mid v_\pi(f_2) }} \pi^{ v_\pi(f_1) } .\]

Let \[g\left( \left(\frac{ \cdot}{ f_2} \right)_\chi \right) = \sum_{ h \in \mathbb F_q[t]/f_{2,\flat}} \left(\frac{h}{f_2}\right)_{\chi} \psi \left( \operatorname{res} \left(\frac{h}{f_{2,\flat}}\right) \right) .\] 

\begin{lemma}\label{gauss-chinta-mohler} We have \begin{equation}\label{eq-gauss-chinta-mohler} g_\chi( f_1,f_2) \frac{1}{ q^{\deg f_2/2} } =  b(f_1,f_2) g\left( \left(\frac{ \cdot}{ f_2} \right)_\chi \right) \left( \frac{ \hat{f}_1} { f_{2,0}} \right)^{-1}_{\chi} \frac{1} {  q^{ \deg f_{2,\flat}/2}}  .\end{equation} \end{lemma}
 
 \begin{proof} We define $\hat{f}_2$ as the largest divisor of $f_2$ prime to $f_{2,0}$, in other words \[ \hat{f}_2= \prod_{\substack{ \pi \textrm{ prime } \\  \pi \mid f_2  \\ \nchi \mid v_\pi(f_2) } }\pi^{ v_\pi(f_2) } ,\]
 as well as $\check{f}_1= f_1/ \hat{f}_1$ and $\check{f}_2 = f_2/\hat{f}_2$.  It is immediate from the definitions that $\hat{f}_2$ is an $\nchi$th power, and that $\hat{f}_1$ and $\hat{f}_2$ are coprime to $\check{f}_1$ and $\check{f}_2$.
 
  By Lemma \ref{gauss-multiplicative} we have
 \[ g_\chi(f_1,f_2) = g_\chi(\hat{f}_1, \hat{f}_2) g_\chi(\check{f}_1,\check{f}_2)  \left( \frac{\hat{f}_2}{ \check{f}_2} \right)_{\chi}  \left( \frac{\check{f}_2}{ \hat{f}_2} \right)_{\chi}  \left( \frac{\hat{f}_1}{\check{ f}_2} \right)^{-1}_{\chi} \left( \frac{\check{f}_1}{ \hat{f}_2} \right)^{-1}_{\chi} .\]

Because $\hat{f}_2$ is an $\nchi$th power and prime to $\check{f}_1$ and $\check{f}_2$, we may ignore all the residue symbols involving $\hat{f}_2$, obtaining 
\begin{equation}\label{gcm-left-split} g_\chi(f_1,f_2) =  g_\chi(\hat{f}_1, \hat{f_2}) g_\chi(\check{f}_1,\check{f}_2)   \left( \frac{\hat{f}_1}{ \check{f}_2} \right)^{-1}_{\chi} \end{equation} 
Similarly, we split the right side of \eqref{eq-gauss-chinta-mohler} into $\hat{f}$ and $\check{f}$ parts. We note that $f_{2,0}$ is also $\check{f}_2$ divided by the greatest $\nchi$th power that divides $\check{f}_2$, in other words, $f_{2,0} = \check{f}_{2,0}$, so that $f_{2,\flat} = \check{f}_{2,\flat}$ and
\begin{equation}\label{gcm-right-split-g} g\left( \left(\frac{ \cdot}{ f_2} \right)_\chi \right) = g\left( \left(\frac{ \cdot}{ \check{f}_2} \right)_\chi \right).\end{equation}
The multiplicativity of $b$ gives
\begin{equation}\label{gcm-left-split-b}  b(f_1,f_2)  =  b(\hat{f}_1,\hat{f}_2)   b(\check{f}_1,\check{f}_2)  .\end{equation}

Combining \eqref{gcm-left-split}, \eqref{gcm-right-split-g}, and \eqref{gcm-left-split-b}, we see that \eqref{gauss-chinta-mohler} is equivalent to
\[ g_\chi(\hat{f}_1, \hat{f_2}) g_\chi(\check{f}_1,\check{f}_2)  \left( \frac{\hat{f}_1}{ \check{f}_2} \right)^{-1}_{\chi}  \frac{1}{ q^{ \deg \hat{f}_2/2 + \deg \check{f}_2/2}}=   b(\hat{f}_1,\hat{f}_2)   b(\check{f}_1,\check{f}_2  )g\left(\left(\frac{ \cdot}{ \check{f}_2} \right)_\chi \right) \left( \frac{ \hat{f}_1} {{f}_{2,0}} \right)^{-1}_{\chi} \frac{1} {  q^{ \deg \check{f}_{2,\flat}/2}} \]
and therefore would follow from the triple of equations
\begin{equation}\label{gcm-hat} g_\chi(\hat{f}_1, \hat{f}_2)  \frac{1}{ q^{ \deg \hat{f}_2/2}}= b(\hat{f}_1,\hat{f}_2)  
\end{equation}
\begin{equation}\label{gcm-mid}  \left( \frac{\hat{f}_1}{ \check{f}_2} \right)^{-1}_{\chi}=\left( \frac{ \hat{f}_1} {{f}_{2,0}} \right)^{-1}_{\chi}\end{equation}
\begin{equation}\label{gcm-check}  g_\chi(\check{f}_1,\check{f}_2)   \frac{1}{ q^{ \deg \check{f}_2/2}} = b(\check{f}_1,\check{f}_2)g\left(\left(\frac{ \cdot}{ \check{f}_2} \right)_\chi \right) \frac{1} {  q^{ \deg \check{f}_{2,\flat}/2}} .
\end{equation}

We now verify these three equations. \eqref{gcm-mid} follows immediately from the fact that ${f}_{2,0}$ and $\check{f}_2$ differ by an $n$th power prime to $\hat{f}_1$.

For \eqref{gcm-hat}, we note from Lemma \ref{gauss-multiplicative} that $g_\chi(f_1,f_2)$ is multiplicative when restricted to $f_2$ that are $n$th powers. Since both sides are multiplicative when restricted to this set, we can reduce to the case that $f_1$ and $f_2$ are prime powers (because any $n$th power can be factored into prime powers that are $n$th powers). In this case, it follows from the definition of $b$ and Lemma \ref{gauss-prime-powers}, noting that $G(\chi^{d_2}, \psi)=-1$ if $d_2$ is divisible by $n$.

For \eqref{gcm-check}, we note that $v_\pi ( \check{f}_2)$ is never a multiple of $n$ for any $\pi$ dividing $\check{f}_2$. It follows from this and the definition of $b$ that $b(\check{f}_1,\check{f}_2)$ vanishes unless $v_\pi (\check{f}_1) = v_\pi(\check{f}_2)-1$ for each such $\pi$. In other words, the right side of \eqref{gcm-check} vanishes unless $\check{f}_1=\check{f}_2/\check{f}_{2,\flat}$. From Lemmas \ref{gauss-multiplicative} and \ref{gauss-prime-powers}, we see that $g(\check{f}_1,\check{f}_2)$ vanishes under the same condition.

Thus, we may assume that $\check{f}_1=\check{f}_2/\check{f}_{2,\flat}$. In this case, \begin{equation}\label{b-simplification} b(\check{f}_1,\check{f}_2 ) = q^{ (\deg \check{f}_2- \deg \check{f}_{2,\flat})/2}\end{equation} since only the second-to-last case of the definition of $b$ occurs. Furthermore, we have 
\[ g_\chi(\check{f}_1,\check{f}_2)  =\sum_{ h \in \mathbb F_q[t] / \check{f}_2}  \left( \frac{h}{ \check{f}_2 } \right)_\chi  \psi \left( \operatorname{res} \left(\frac{h \check{f}_{1} }{\check{f}_{2}}\right) \right) \] \[= \sum_{ h \in \mathbb F_q[t] / \check{f}_2}  \left( \frac{h}{\check{f}_2 } \right)_\chi   \psi \left( \operatorname{res} \left(\frac{h}{\check{f}_{2,\flat}}\right) \right)=
 g\left(\left(\frac{ \cdot}{ \check{f}_2} \right)_\chi \right) q^{ \deg \check{f}_2 - \deg \check{f}_{2,\flat}}\]
 which together with \eqref{b-simplification} gives \eqref{gcm-check}.
 \end{proof}

\begin{proof}[Proof of Proposition \ref{gauss-example} using Chinta-Mohler]
Let  \[ \tilde{a}^* (f_1,f_2; q ,\chi , M ) = G(\chi,\psi)^{\deg f_2} \tilde{a}(f_1,f_2;q,\chi,M)=   (-1)^{ \frac{ \deg f_2 (\deg f_2-1) (q-1)}{4}}     \sum_{ \substack{ u \in \mathbb F_q[t]^+ \\ u^n | f_2}}  q^{ (n-1)  \deg u} g_\chi(f_1, f_2/u^n ) .\] 

We prove that $\tilde{a}^*$ satisfies the axioms of Corollary \ref{axiomatics-2} with $w_1=\epsilon_1=0$, $w_2=\epsilon_2=1$, $\gamma_1(q,\chi)=1$, $\gamma_2(q,\chi) =G(\chi,\psi)$. 

For axiom (1) we have \begin{equation}\label{gauss-example-cm-1-start} \begin{split} & \tilde{a}^* (f_1f_3, f_2f_4; q,\chi,M )\\ =& (-1)^{ \frac{ \deg f_2 (\deg f_2-1) (q-1)}{4}+ \frac{ \deg f_4 (\deg f_4-1) (q-1)}{4} + \frac{ \deg f_2 \deg f_4(q-1)}{2} }    \sum_{ \substack{ u \in \mathbb F_q[t]^+ \\ u^n | f_2f_4}}  q^{ (n-1)  \deg u} g_\chi(f_1f_3, f_2f_4/ u^n)  \end{split} \end{equation}

Because $f_2$ and $f_4$ are coprime, we can write any $u$ where $u^n\mid f_2f_4$ uniquely as $u_2u_4$ where $u_2^n$ divides $f_2$ and $u_4^n$ divides $f_4$

From Lemma \ref{gauss-multiplicative} we get
\[ g_\chi(f_1f_3, f_2f_4/ (u_2^n u_4^n) )  = g_\chi(f_1,f_2/u_2^n) g_\chi(f_3,f_4/u_4^n) \left(\frac{f_2/u_2^n}{f_4/u_4^n} \right)_\chi \left(\frac{f_4/u_4^n}{f_2/u_2^n} \right)_\chi  \left( \frac{f_1}{f_4/u_4^n} \right)_\chi^{-1} \left(\frac{f_3}{f_2/u_2^n} \right)_\chi^{-1} .\]

However, we can ignore the $u_2^n$ and $u_4^n$ factors in the power residue symbols as they are $\nchi$th powers and because $u_2$, dividing $f_2$, is prime to $f_3$ and $f_4$ and similarly $u_4$ is prime to $f_1$ and $f_2$. Thus
\begin{equation}\label{gauss-multiplicative-cm-refined} g_\chi(f_1f_3, f_2f_4/ (u_2^n u_4^n) )  = g_\chi(f_1,f_2/u_2^n) g_\chi(f_3,f_4/u_4^n) \left(\frac{f_2}{f_4} \right)_\chi \left(\frac{f_4}{f_2} \right)_\chi  \left( \frac{f_1}{f_4} \right)_\chi^{-1} \left(\frac{f_3}{f_2} \right)_\chi^{-1} .\end{equation}

Plugging \eqref{gauss-multiplicative-cm-refined} into \eqref{gauss-example-cm-1-start} gives \begin{equation}\label{gauss-example-cm-1-middle} \begin{split}&  \tilde{a}^*(f_1f_3, f_2f_4; q,\chi,M ) \\ = & \tilde{a}^*(f_1,f_2; q,\chi,M) \tilde{a}^*(f_3,f_4;q,\chi,M)  (-1)^{\frac{ \deg f_2 \deg f_4 (q-1)}{2}}  \left(\frac{f_2}{f_4} \right)_\chi \left(\frac{f_4}{f_2} \right)_\chi  \left( \frac{f_1}{f_4} \right)_\chi^{-1} \left(\frac{f_3}{f_2} \right)_\chi^{-1}.\end{split} \end{equation}

We have \[(-1)^{\frac{ \deg f_2 \deg f_4(q-1)}{2}}  = \left(\frac{f_2}{f_4} \right)_\chi^{n/2} \left(\frac{f_4}{f_2} \right)_\chi ^{n/2} \] by Lemma \ref{reciprocity}, which, plugged into \eqref{gauss-example-cm-1-middle}, verifies axiom (1).

For axiom (2), we have
\[ \tilde{a}^*(T-x, 1; q, \chi, M)=g_\chi(T-x,1 ) = 1\]
and
\[ \tilde{a}^*(1,T-x;q,\chi,M) =   g_\chi( 1, T-x) =G(\chi,\psi),\]
both using Lemma \ref{gauss-prime-powers}.

Next, let 
\[ J_1( d_1,d_2; q, \chi, M) = \begin{cases}\{ (1,1)\} & \textrm{ if } d_2=0 \\
\{ (q^{d_2},1) ,( q^{ (d_2-1) } , -1) \} & \textrm{ if }d_2 \equiv 0 \bmod n\textrm{ and }d_1 \geq d_2 \\
\emptyset  & \textrm{ if }d_2 \not\equiv 0 \bmod n\textrm{ and }d_1 \geq d_2 \\
\{ (  - q^{(d_2-1 ) }  G(\chi^{d_2} ,\psi) , -1) \} & \textrm{ if } d_1 = d_2 -1 \\
\emptyset  & \textrm{if } d_1< d_2 -1 \end{cases} \]

Then by Lemma \ref{gauss-prime-powers} we have 
\[ g_\chi(\pi^{d_1}, \pi^{d_2}) =\left(\frac{\pi'}{\pi}\right)^{d_2}_\chi  \sum_{j \in J_1(d_1,d_2; q,\chi,M)}  c_j \alpha_j^{ \deg \pi} \]
noting that the $\left(\frac{\pi'}{\pi}\right)^{d_2}_\chi$ term can be ignored in the cases where $d_2$ is divisible by $\nchi$.

Furthermore, we have
\[ \sum_{ \substack{ u \in \mathbb F_q[t]^+ \\ u^n | \pi^{d_2} }}  q^{ (n-1)  \deg u} g_\chi(\pi^{d_1}, \pi^{d_2}/u^n ) = \sum_{ c =0}^{\lfloor d_2/n\rfloor} q^{ (n-1) c \deg \pi} g_\chi(\pi^{d_1}, \pi^{d_2 -nc} ).  \]

So letting 
\[ J(d_1,d_2;q,\chi,M) =(-1)^{ \frac{ d_2 (d_2-1) (q-1)}{4}} \bigcup_{c=0}^{\lfloor d_2/n\rfloor } q^{(n-1) c}  J_1(d_1,d_2-nc;q,\chi, m)\]
 we have
\[  \tilde{a}^* (\pi^{d_1}, \pi^{d_2}; q,\chi,M) \] \[ = (-1)^{ \frac{ d_2 \deg \pi (d_2 \deg \pi -1) (q-1)}{4}} \sum_{ \substack{ w \in \mathbb F_q[t]^+ \\ u^n | \pi^{d_2} }}  q^{ (n-1)  \deg u} g_\chi(\pi^{d_1}, \pi^{d_2}/u^n ) \] \[  = (-1)^{ \frac{ d_2 \deg \pi (d_2 \deg \pi -1) (q-1)}{4}}   (-1)^{ \frac{ \deg \pi d_2 (d_2-1) (q-1)}{4}}  
\left(\frac{\pi'}{\pi}\right)^{d_2}_\chi  \sum_{j \in J(d_1,d_2; q,\chi,M) } c_j \alpha_j^{ \deg \pi} \]

verifying axiom (3) because $d_2 \equiv d_2^2\mod 2$ and thus \[ 
 (-1)^{ \frac{ d_2 \deg \pi (d_2 \deg \pi -1) (q-1)}{4}}   (-1)^{ \frac{ \deg \pi d_2 (d_2-1) (q-1)}{4}} = (-1)^{ \frac{d_2 \deg \pi  (d_2 \deg \pi - d_2) (q-1) } {4}} \] 
 \[ = (-1)^{ \frac{d_2^2 \deg \pi (\deg \pi-1) (q-1) }{4}}= (-1)^{ \frac{ d_2  \deg \pi ( \deg \pi -1) (q-1)}{4}} = \left(\frac{\pi'}{\pi} \right)^{d_2 (n/2) }_{\chi}  (-1)^{ d_2 (\deg \pi + 1)} \] by Lemma \ref{mobius-evaluation}.

Next, to verify axiom (4), it suffices to show that
\begin{equation}\label{gauss-example-cm-4-start} \begin{split} (-1)^{ \frac{ d_2 (d_2-1) (q-1)}{4}}   &  \sum_{\substack { f_1,f_2 \in \mathbb F_q[t]^+ \\ \deg f_1 =d_1, \deg f_2=d_2}}  \sum_{ \substack{ u \in \mathbb F_q[t]^+ \\ u^n | f_2}}  q^{ (n-1)  \deg u} g_\chi(f_1, f_2/u^n) 
\\=&   \sum_{j \in J(d_1,d_2; q,\chi,M)} c_j \frac{ q^{d_1 + 2 d_2}}{ \overline{\alpha}_j} .
\end{split}\end{equation}

To do this, it suffices to show the identity of formal power series in $q^{-s}$ and $q^{ - (w+1/2)}$ 
\begin{equation}\label{gauss-example-cm-4-next} \begin{split} &  \sum_{\substack f_1,f_2 \in \mathbb F_q[t]^+ }  \sum_{ \substack{ u \in \mathbb F_q[t]^+ \\ u^n | f_2}}  q^{ (n-1)  \deg u} g_\chi(f_1, f_2/u^n) q^{-s \deg f_1} q^{-(w+1/2)  \deg f_2}  \\  =&  \sum_{d_1,d_2} q^{ -s d_1} q^{-(w+1/2)  d_2}  \sum_{j \in J(d_1,d_2; q,\chi,M)}  (-1)^{ \frac{ d_2 (d_2-1) (q-1)}{4}}\frac{ q^{d_1 + 2d_2 }}{ \overline{\alpha}_j} .\end{split}\end{equation}

The change of variables $f_2 \mapsto u^n f_2$, the identity $\zeta(nw-n/2+1) = \sum_{ u \in \mathbb F_q[t]^+} q^{ ( -nw + n/2-1) \deg u} $, and finally Lemma \ref{gauss-chinta-mohler}, together transform the left side of  \eqref{gauss-example-cm-4-next} into
\begin{equation}\label{compare-to-cm}  \begin{split}&  \sum_{\substack f_1,f_2 \in \mathbb F_q[t]^+ }  \sum_{ \substack{ u \in \mathbb F_q[t]^+ }}  q^{ (n-1)  \deg u} g_\chi(f_1, f_2) q^{-s \deg f_1} q^{-(w+1/2)  (\deg f_2+ n \deg u ) }   \\
& = \zeta (nw - n/2 + 1) \sum_{\substack f_1,f_2 \in \mathbb F_q[t]^+ }  \frac{ g_\chi(f_1, f_2)}{ q^{  \deg f_2/2}} q^{ -s \deg f_1}  q^{ -w\deg f_2} \\
&= \zeta (nw - n/2 + 1) \sum_{\substack f_1,f_2 \in \mathbb F_q[t]^+ }  b(f_1,f_2) g\left( \left(\frac{ \cdot}{ f_2} \right)_\chi \right) \left( \frac{ \hat{f}_1} { f_{2,0}} \right)^{-1}_{\chi} \frac{1} {  q^{ \deg f_{2,\flat}/2}}  q^{ -s \deg f_1}  q^{ -w\deg f_2}  = Z_2(s,w) \end{split} \end{equation}
as defined in \citep[(1.6)]{ChintaMohler}.

The definition of $J$ gives
\[\sum_{j \in J(d_1,d_2; q,\chi,M)}  (-1)^{ \frac{ d_2 (d_2-1) (q-1)}{4}} c_j \frac{ q^{d_1 + 2d_2 }}{ \overline{\alpha}_j}  = \sum_{c=0}^{\lfloor d_2/n\rfloor }  \sum_{j \in J_1(d_1,d_2-nc; q,\chi,M)} \frac{ q^{d_1 + 2d_2 }}{ q^{(n-1) c} \overline{\alpha}_j}.\] This, followed by the substitution $d_2 \mapsto d_2+nc$ and the evaluation of a geometric series in $c$, implies that the right side of \eqref{gauss-example-cm-4-next} is equal to \[ \sum_{d_1,d_2} q^{ -s d_1} q^{-(w+1/2)  d_2} \sum_{c=0}^{\lfloor d_2/n\rfloor }  \sum_{j \in J_1(d_1,d_2-nc; q,\chi,M)} \frac{ q^{d_1 + 2d_2 }}{ q^{(n-1) c} \overline{\alpha}_j}\]
\[ = \sum_{d_1,d_2}  \sum_{c=0}^\infty q^{-s d_1} q^{ - (w+1/2) (d_2+nc)} \sum_{j \in J_1(d_1,d_2; q,\chi,M)} \frac{ q^{d_1 + 2d_2 + 2nc }}{ q^{(n-1) c} \overline{\alpha}_j}\] \[ =\frac{1}{1-  q^{ - n (w+1/2) } \frac{q^{ 2n}}{ q^{n-1}}} \sum_{d_1,d_2} q^{ -s d_1} q^{-(w+1/2)  d_2}  \sum_{j \in J_1 (d_1,d_2; q,\chi,M)} c_j \frac{ q^{d_1 + 2d_2 }}{ \overline{\alpha}_j}.\]

We have \begin{equation}\label{gauss-example-cm-4-if} \frac{1}{1-  q^{ - n (w+1/2) } \frac{q^{ 2n}}{ q^{n-1}}}   = \frac{1}{ 1- q^{ \frac{n}{2} +1 - nw}} \end{equation} 
and \[  \sum_{d_1,d_2} q^{ -s d_1} q^{-(w+1/2)  d_2}   \sum_{j \in J_1 (d_1,d_2; q,\chi,M)} c_j\frac{ q^{d_1 + 2d_2 }}{ \overline{\alpha}_j} \]
\[=  \sum_{\substack{ d_1 \in \mathbb N\\ d_2 =0} }  q^{d_1 -d_1 s}  + \sum_{ \substack { d_2 \in \mathbb N^+\\  d_2 \equiv 0 \bmod n\\ d_1 \geq d_2}}   q^{-  s d_1 - w d_2}  q^{d_1+ d_2/2} (1-q)  \] \[+ \sum_{i=1}^{n-1} \sum_{ \substack { d_2 \in \mathbb N\\  d_2 \equiv i \bmod n\\ d_1 = d_2-1}} q^{ -s d_1 - w d_2}  q^{ d_1 + d_2/2} G(\chi^i, \psi) -  \sum_{ \substack { d_2 \in \mathbb N\\  d_2 \equiv 0\bmod n\\ d_1 = d_2-1}} q^{ -s d_1 - w d_2}  q^{ d_1 + d_2/2+1 } \]
\begin{equation}\label{gauss-sum-cm-4-prexy}= \frac{1}{ 1- q^{1-s}} + \frac{1}{1 - q^{1-s}}  \frac{ (1-q) q^{ -n s - n w + 3n/2} }{ 1- q^{-ns - nw +3n/2}} + \sum_{i=1}^{n-1}  \frac{ q^{ - (i-1) s - i w  + 3i/2 - 1} G(\chi^i,\psi)     }{1 - q^{-ns-nw+ 3n/2}}   - \frac{  q^{ - (n-1) s - n w + 3n/2}        }{1 - q^{-ns-nw+ 3n/2}}.\end{equation}

Introducing the variables $x = q^{-s}$ and $y=q^{-w}$, we can rewrite \eqref{gauss-sum-cm-4-prexy} as
\[ \frac{1}{ 1- qx} + \frac{1}{1 - qx}  \frac{ (1-q) q^{  3n/2} x^ny^n }{ 1- q^{3n/2} x^n y^n } + \sum_{i=1}^{n-1}  \frac{ q^{3i/2 - 1} G(\chi^i,\psi) x^{i-1} y^i  } { 1- q^{3n/2} x^n y^n}  -   \frac{ q^{  3n/2}   x^{n-1} y^n    }{1 -q^{3n/2} x^n y^n }\] 
\[=   \frac{1 - q^{3n/2+1} x^n y^n }{ (1- qx) (1- q^{3n/2} x^n y^n ) } + \sum_{i=1}^{n-1}  \frac{ q^{ 3i/2 - 1} G(\chi^i,\psi) x^{i-1} y^i  } { 1- q^{3n/2} x^n y^n}  -   \frac{ q^{  3n/2}   x^{n-1} y^n    }{1 -q^{3n/2} x^n y^n }\] 
\[=   \frac{1 - q^{3n/2+1} x^n y^n  + \sum_{i=1}^{n-1} q^{3i/2 -1} G(\chi^i,\psi) x^{i-1} y^i  (1-qx) -  q^{3n/2} x^{n-1} y^n (1-qx)   }{ (1- qx) (1- q^{3n/2} x^n y^n ) } \]
\[=   \frac{1  -  q^{3n/2} x^{n-1} y^n  +   \sum_{i=1}^{n-1}  q^{ 3i/2 - 1} G(\chi^i,\psi) x^{i-1} y^i  (1-qx)   }{ (1- qx) (1- q^{3n/2} x^n y^n ) } .\]

So bringing in the initial factor \eqref{gauss-example-cm-4-if}, \eqref{gauss-example-cm-4-next} is equivalent to
\begin{equation}\label{gauss-example-cm-4-final}  Z_2 = \frac{1  -  q^{3n/2} x^{n-1} y^n  +   \sum_{i=1}^{n-1} q^{ 3i/2 - 1} G(\chi^i,\psi) x^{i-1} y^i  (1-qx)   } {  (1- q^{\frac{n}{2}+1} y^n ) (1- qx) (1- q^{3n/2} x^n y^n ) }.\end{equation}

Noting that $\tau(\epsilon^i) = G(\chi^i,\psi)$, \eqref{gauss-example-cm-4-final} is precisely \cite[(1.8)]{ChintaMohler}, finishing the proof of axiom (4).

For axiom (5), we first check that $J_1(d_1,d_2; q,\chi, M)$ has all $|\alpha_j|< q^{ \frac{d_1+2d_2-1}{2}} $ unless $(d_1,d_2)=(0,0)$ or $(0,1)$, case-by case. In the $d_2 \equiv 0 \mod n $ and $d_1 \geq d_2$ case, the key is that $d_1 \geq d_2 \geq n \geq 2$ so $ q^{ \frac{d_1+2d_2-1}{2} }> q^{d_2}$, and in the $d_1=d_2-1$ case, we have $q^{ \frac{d_1+2d_2-1}{2}} > q^{ d_2 -\frac{1}{2}}$ as long as $d_1>0$. Furthermore, in the $(0,0)$ and $(0,1)$ cases, we have $|\alpha_j| \leq q^{ \frac{d_1+2d_2}{2}}$.

By the definition of $J$ in terms of $J_1$, it follows that each $\alpha_j$ appearing either has $c=0$ and thus satisfies $|\alpha_j| < q^{ \frac{d_1+2d_2-1}{2}} $ since $d_1+d_2\geq 2$ implies $(d_1,d_2)\neq (0,0),(0,1)$, or has $c>0$ in which case $|\alpha_j|\leq  q^{(n-1)c} q^{ \frac{d_1+2(d_2-nc) }{2}} = q^{ \frac{ d_1 + 2 d_2 - 2c}{2}} <  q^{ \frac{d_1+2d_2-1}{2}}$ since $2c \geq 2>1$, verifying (5).\end{proof}

We now describe a third case where we can relate $a(f_1,f_2; q,\chi,M)$ to prior work. First, following  \cite[(3.2),(3.3)]{Chinta}, let $H(f_1,f_2)$ be the unique function satisfying 

\begin{enumerate}

\item If $\gcd(f_1f_2,g_1g_2)=1$ then \[ H(f_1g_1,f_2g_2) = \left( \frac{f_1}{g_1} \right)_{\chi}  \left( \frac{g_1}{f_1} \right)_{\chi}   \left( \frac{f_2}{g_2} \right)_{\chi}  \left( \frac{g_2}{f_2} \right)_{\chi}   \left( \frac{f_1}{g_2} \right)_{\chi} ^{-1}   \left( \frac{g_1}{f_2} \right)_{\chi} ^{-1} H(f_1,f_2) H(g_1,g_2). \]

\item For $\pi$ prime,

\[ H(\pi^{d_1},\pi^{d_2} )= \begin{cases} 
1 & \textrm{if } (d_1,d_2)=(0,0)\\
g_\chi(1,\pi) &\textrm{if } (d_1,d_2)=(1,0) \textrm{ or }(0,1)\\
g_\chi(\pi,\pi^2)g_\chi(1,\pi) & \textrm{if } (d_1,d_2)=(2,1)\textrm{ or }(1,2) \\
g_\chi(\pi,\pi^2)g_\chi(1,\pi)^2 & \textrm{if } (d_1,d_2)=(2,2) \\
0 & \textrm{otherwise} \\ \end{cases}\]

\end{enumerate}

\begin{prop}\label{double-gauss-example} Assume $\nchi$ even and $q \equiv 1\bmod 4$. 

 Take $r=2$, $M= \begin{pmatrix} \frac{n}{2}+1 & -1 \\ -1 & \frac{n}{2}+1 \end{pmatrix} $. 

Then \[ a (f_1,f_2; q ,\chi , M ) =   \frac{1  }{  G(\chi,\psi)^{\deg f_1+\deg f_2 }}   \sum_{ \substack{ a,b,c \in \mathbb F_q[t]^+ \\ a^nb^n  | f_1\\ b^n c^n | f_2 }}  q^{ (n-1)  \deg a + (2n-1) \deg b + (n-1)\deg c} H (f_1/a^nb^n, f_2/b^nc^n) .\] 

\end{prop}

\begin{proof} To prove this, we verify the axioms of Theorem \ref{axiomatics-2} are satisfied for \[ \tilde{a}^* (f_1,f_2; q ,\chi , M ) =  \sum_{ \substack{ a,b,c \in \mathbb F_q[t]^+ \\ a^nb^n  | f_1\\ b^n c^n |    f_2 }}  q^{ (n-1)  \deg a + (2n-1) \deg b + (n-1)\deg c} H (f_1/a^nb^n, f_2/b^nc^n) \] with $\epsilon_1=\epsilon_2 = 1$, $w_1=w_2=1$, $\gamma_1(q,\chi)=\gamma_2(q,\chi) = G(\psi,\chi)$. 

The multiplicativity axiom (1) follows immediately from the multiplicativity axiom of $H$, noting that the factors $a^nb^n$ and $b^nc^n$ have degree divisible by $n$ and can be ignored, and that the term $\left( \frac{f_1}{g_1} \right)_\chi^{n/2} \left( \frac{g_1}{f_1} \right)_\chi^{n/2}$ is $1$ by Lemma \ref{reciprocity}, because $q \equiv 1\bmod 4$, and so can be ignored.

Axiom (2) is straightforward. In the case when $(f_1,f_2)=(T-x,1)$ or $(1,T-x)$, the sum over $a,b,c$ is trivial, and $H(f_1,f_2)= g_\chi(1,\pi) = G(\chi,\psi)$.

We have that
\begin{equation}\label{dge-3} \tilde{a}^* ( \pi^{d_1}, \pi^{d_2} ; q, \chi, M) =  \sum_{ \substack{ j_1,j_{12},j_2 \in \mathbb N\\ n(j_1 +j_{12} ) \leq d_1 \\ n(j_{12} + j_2) \leq d_2 } }  q^{((n-1) j_1 + (2n-1) j_{12} + (n-1) j_2)\deg \pi }   H( \pi^{d_1 - nj_1 - n j_{12} },  \pi^{d_2 - n j_{12} - nj_2}) . \end{equation}

From Lemma \ref{gauss-prime-powers} and the Hasse-Davenport identities, we have $g_\chi(1,\pi) = -   ( - G(\chi,\psi))^{\deg \pi}  \left( \frac{\pi'}{\pi} \right)_\chi  $ and $g_\chi(\pi,\pi^2) = - (- q G(\chi^2,\psi))^{\deg \pi}  \left( \frac{\pi'}{\pi} \right)_\chi^2$, so we can write \eqref{dge-3} as
\[  \left( \frac{\pi'}{\pi} \right)_\chi^{ d_1+d_2}   \sum_{ \substack{(j_1,j_{12},j_2, r_1,r_2) \in \mathbb N^5 \\  n j_1 +n j_{12} +r_1 = d_1 \\ n j_{12} + n j_2 +r_2 = d_2 \\ (r_1,r_2) \in \{ (0,0), (1,0),(0,1),(2,1),(1,2),(2,2)\} }} c_{ (j_1,j_{12},j_2,r_1,r_2)} \alpha_{(j_1,j_{12} j_2,r_1,r_2)}^{\deg \pi }\]
where
\[  \alpha_{(j_1,j_{12},j_2,r_1,r_2)} =   q^{(n-1) j_1 + (2n-1) j_{12} + (n-1) j_2   }\begin{cases} 
1 & \textrm{if } (r_1,r_2)=(0,0)\\
-G(\chi,\psi) &\textrm{if } (r_1,r_2)=(1,0) \textrm{ or }(0,1)\\
qG(\chi^2,\psi) G(\chi,\psi) & \textrm{if } (r_1,r_2)=(2,1)\textrm{ or }(1,2) \\
-q G(\chi^2,\psi) G(\chi,\psi)^2 & \textrm{if } (r_1,r_2)=(2,2)  \\ \end{cases}\]
and
\[ c_{(j_1,j_{12},j_{2},r_1,r_2)}= \begin{cases} 1 & \textrm{if }(r_1,r_2)=(0,0),(2,1),\textrm{ or }(1,2) \\ -1 & \textrm{if } (r_1,r_2) = (1,0),(0,1),\textrm{ or }(2,2) \\ \end{cases}. \]

So we may take \[J (d_1,d_2;q,\chi,M)= \left\{(j_1,j_{12},j_2, r_1,r_2) \in \mathbb N^5 \mid  \substack{n j_1 +n j_{12} +r_1 = d_1 \\  n j_{12} + n j_2 +r_2 = d_2 \\ (r_1,r_2) \in \{ (0,0), (1,0),(0,1),(2,1),(1,2),(2,2)\}}\right \}\] and take these $\alpha_j$ and $c_j$. By \eqref{me-3}, because $q \equiv 1 \bmod 4$, we have  $(-1)^{ (d_1 + d_2) (\deg \pi +1)} = \left(\frac{\pi'}{\pi}\right)_{\chi}^{ (d_1+d_2) (n/2)}$. This, and the definition of $J$, implies $\tilde{a}^*$ satisfies axiom (3).

 $J$ is a manifestly a compatible system of sets of ordered pairs. For axiom (4), we must check
\[ \sum_{f_1,f_2 \in \mathbb F_q[t]^+} \tilde{a}^* ( f_1,f_2; q,\chi,M) x^{ \deg f_1} y^{\deg f_2}  \] \[=  \sum_{j_1,j_{12},j_2 \in \mathbb N} \sum_{(r_1,r_2) \in \{(0,0),(0,1),(1,0),(1,2),(2,1),(2,2)\}}  c_{(j_1,j_{12}, j_2,r_1,r_2)} \frac{q^{2d_1 +2d_2}} {\overline{\alpha}_{(j_1,j_{12},j_2,r_1,r_2)}}  x^{ n j_1 + nj_{12} + r_1} y^{n j_{12} + nj_2 +r_2}. \]

We have 
\[ \frac{q^{2d_1 +2d_2}} {\overline{\alpha}_{(j_1,j_{12},j_2,r_1,r_2)}}=  q^{ (n+1) j_1 +(2n+1)  j_{12}+ (n+1) j_2 }  \begin{cases} 
1 & \textrm{if } (r_1,r_2)=(0,0)\\
-q G(\chi,\psi) &\textrm{if } (r_1,r_2)=(1,0) \textrm{ or }(0,1)\\
q^3 G(\chi^2,\psi) G(\chi,\psi) & \textrm{if } (r_1,r_2)=(2,1)\textrm{ or }(1,2) \\
-q^4 G(\chi^2,\psi) G(\chi,\psi)^2 & \textrm{if } (r_1,r_2)=(2,2)  \\ \end{cases}.\]
Here we use $G(\chi,\psi) \overline{G(\chi,\psi)}=q$ to calculate the inverse conjugate of $\alpha$. 

Hence we have
\[\sum_{j_1,j_{12},j_2 \in \mathbb N} \sum_{(r_1,r_2) \in \{(0,0),(0,1),(1,0),(1,2),(2,1),(2,2)\}}   c_{(j_1,j_{12}, j_2,r_1,r_2)}\frac{q^{2d_1 +2d_2}} {\overline{\alpha}_{(j_1,j_{12},j_2,r_1,r_2)}}  x^{ n j_1 + nj_{12} + r_1} y^{n j_{12} + nj_2 +r_2} =\]\[ \frac{1+ qG(\chi,\psi) x + q G(\chi,\psi) y +q^3 G(\chi^2,\psi) G(\chi,\psi) x^2 y +q^3 G(\chi^2,\psi) G(\chi,\psi)  xy^2 + q^4 G(\chi^2,\psi) G(\chi,\psi)^2 x^2 y^2 }{ (1 - q^{n+1} x^n) (1- q^{2n+1} x^ny^n ) (1-q^{n+1} x^n y^2 )}.\]

By the definition of the series $Z(x,y)$ in \cite{Chinta}, we have

\[ Z ( x,y) =  \sum_{f_1,f_2 \in \mathbb F_q[t]^+} \tilde{a}^* ( f_1,f_2; q,\chi,M) x^{ \deg f_1} y^{\deg f_2}. \]  According to \cite[Theorem 4.2]{Chinta}, upon observing that $\tau_1= G(\chi,\psi)$ and $\tau_2 =G(\chi^2,\psi)$, we have
\[Z ( x,y) =\] \[ \frac{1+ qG(\chi,\psi) x + q G(\chi,\psi) y +q^3 G(\chi^2,\psi) G(\chi,\psi) x^2 y +q^3 G(\chi^2,\psi) G(\chi,\psi)  xy^2 + q^4 G(\chi^2,\psi) G(\chi,\psi)^2 x^2 y^2 }{ (1 - q^{n+1} x^n) (1- q^{2n+1} x^ny^n ) (1-q^{n+1} x^n y^2 )}\]
which is exactly the desired identity. 

For axiom (5), note that

\[ \log_q |\alpha_{(j_1,j_{12},j_2,r_1,r_2)} | = \left( n-1 \right) j_1  + \left(2 n-1\right) j_{12} + \left(n-1\right) j_2 + \begin{cases} 
0 & \textrm{if } (r_1,r_2)=(0,0)\\
 \frac{1}{2}  &\textrm{if } (r_1,r_2)=(1,0) \textrm{ or }(0,1)\\
2& \textrm{if } (r_1,r_2)=(2,1)\textrm{ or }(1,2) \\
\frac{3}{2}  & \textrm{if } (r_1,r_2)=(2,2)  \\ \end{cases}\]
which is $< n j_1 + 2n j_{12} + n j_2 +r_1 + r_2-\frac{1}{2} $ as long as $n j_1 + 2n j_{12} + n j_2 +r_1 + r_2 \geq 2$.

\end{proof}

\bibliographystyle{plainnat}

\bibliography{references}

\end{document}